\documentclass[11pt,a4paper]{amsart}[2000]
\usepackage{geometry}
\usepackage{amsmath}
\usepackage{amssymb}
\usepackage{amsthm}
\usepackage{mathrsfs}
\usepackage[hang]{footmisc}
\usepackage[all]{xypic}
\usepackage{hyperref}
\usepackage{color}
\usepackage{mathtools}
\usepackage{enumitem}
\usepackage{stackengine}

\thanks{}

\allowdisplaybreaks

\theoremstyle{plain}

\newtheorem{Thm}{Theorem}[section]

\theoremstyle{definition}

\theoremstyle{plain}
\newtheorem{thm}[Thm]{Theorem}
\newtheorem{lem}[Thm]{Lemma}
\newtheorem{cor}[Thm]{Corollary}
\newtheorem{prop}[Thm]{Proposition}

\theoremstyle{definition}
\newtheorem{defn}[Thm]{Definition}

\newtheorem{eg}[Thm]{Example}
\newtheorem{rmk}[Thm]{Remark}

\newtheorem{innercustomthm}{Theorem}
\newenvironment{customthm}[1]
{\renewcommand\theinnercustomthm{#1}\innercustomthm}
{\endinnercustomthm}

\newcommand{\B}{B}
\newcommand{\A}{A}
\newcommand{\J}{J}
\newcommand{\K}{\mathcal{K}}
\newcommand{\D}{D}
\newcommand{\Ch}{D}
\newcommand{\Zh}{\mathcal{Z}}
\newcommand{\E}{E}
\newcommand{\Oh}{\mathcal{O}}

\newcommand{\T}{{\mathbb T}}
\newcommand{\R}{{\mathbb R}}
\newcommand{\N}{{\mathbb N}}
\newcommand{\Z}{{\mathbb Z}}
\newcommand{\C}{{\mathbb C}}
\newcommand{\Q}{{\mathbb Q}}

\newcommand{\la}{{\langle}}
\newcommand{\ra}{{\rangle}}

\newcommand{\aut}{\mathrm{Aut}}
\newcommand{\supp}{\mathrm{supp}}

\newcommand{\eps}{\varepsilon}
\numberwithin{equation}{section}

\newcommand{\id}{\mathrm{id}}

\newcommand{\halpha}{\widehat{\alpha}}
\newcommand{\calpha}{\widehat{\alpha}}

\newcommand{\tih}{\widetilde {h}}

\newcommand\set[1]{\left\{#1\right\}}  
\newcommand\mset[1]{\left\{\!\!\left\{#1\right\}\!\!\right\}}

\newcommand{\CA}[0]{\mathcal{A}} \newcommand{\CB}[0]{\mathcal{B}}
\newcommand{\CC}[0]{\mathcal{C}} \newcommand{\CD}[0]{\mathcal{D}}
 
\newcommand{\CG}[0]{\mathcal{G}} \newcommand{\CH}[0]{\mathcal{H}}

\newcommand{\CO}[0]{\mathcal{O}} 
\newcommand{\CQ}[0]{\mathcal{Q}} \newcommand{\CR}[0]{\mathcal{R}}
 \newcommand{\CT}[0]{\mathcal{T}}
\newcommand{\CU}[0]{\mathcal{U}}

\newcommand{\Ra}[0]{\Rightarrow}
\newcommand{\La}[0]{\Leftarrow}
\newcommand{\LRa}[0]{\Leftrightarrow}

\newcommand{\quer}[0]{\overline}
\newcommand{\eins}[0]{\mathbf{1}}			
\newcommand{\diag}[0]{\operatorname{diag}}
\newcommand{\ad}[0]{\operatorname{Ad}}
\newcommand{\ev}[0]{\operatorname{ev}}
\newcommand{\fin}[0]{{\subset\!\!\!\subset}}
\newcommand{\diam}[0]{\operatorname{diam}}
\newcommand{\Hom}[0]{\operatorname{Hom}}
\newcommand{\dst}[0]{\displaystyle}
\newcommand{\spp}[0]{\operatorname{supp}}
\newcommand{\lsc}[0]{\operatorname{Lsc}}
\newcommand{\del}[0]{\partial}
\newcommand{\GU}[0]{\CG^{(0)}}

\newtheorem{theorem}[Thm]{Theorem}

\theoremstyle{definition}

\numberwithin{equation}{Thm}

\title[Boundary actions]{Boundary actions of $\cat$ spaces and their $C^*$-algebras}	

\begin{document}
\global\long\def\floorstar#1{\lfloor#1\rfloor}
\global\long\def\ceilstar#1{\lceil#1\rceil}	

\global\long\def\B{B}
\global\long\def\A{A}
\global\long\def\J{J}
\global\long\def\K{\mathcal{K}}
\global\long\def\D{D}
\global\long\def\Ch{D}
\global\long\def\Zh{\mathcal{Z}}
\global\long\def\E{E}
\global\long\def\Oh{\mathcal{O}}

\global\long\def\T{{\mathbb{T}}}
\global\long\def\BR{{\mathbb{R}}}
\global\long\def\N{{\mathbb{N}}}
\global\long\def\Z{{\mathbb{Z}}}
\global\long\def\C{{\mathbb{C}}}
\global\long\def\Q{{\mathbb{Q}}}

\global\long\def\aut{\mathrm{Aut}}
\global\long\def\supp{\mathrm{supp}}

\global\long\def\eps{\varepsilon}

\global\long\def\id{\mathrm{id}}

\global\long\def\halpha{\widehat{\alpha}}
\global\long\def\calpha{\widehat{\alpha}}

\global\long\def\tih{\widetilde{h}}

\global\long\def\opFol{\operatorname{F{\o}l}}

\global\long\def\opRange{\operatorname{Range}}

\global\long\def\opIso{\operatorname{Iso}}

\global\long\def\dimnuc{\dim_{\operatorname{nuc}}}

\global\long\def\set#1{\left\{  #1\right\}  }


\global\long\def\mset#1{\left\{  \!\!\left\{  #1\right\}  \!\!\right\}  }

\global\long\def\Ra{\Rightarrow}
\global\long\def\La{\Leftarrow}
\global\long\def\LRa{\Leftrightarrow}

\global\long\def\quer{\overline{}}
\global\long\def\eins{\mathbf{1}}
\global\long\def\diag{\operatorname{diag}}
\global\long\def\ad{\operatorname{Ad}}
\global\long\def\ev{\operatorname{ev}}
\global\long\def\fin{{\subset\!\!\!\subset}}
\global\long\def\diam{\operatorname{diam}}
\global\long\def\Hom{\operatorname{Hom}}
\global\long\def\dst{{\displaystyle }}
\global\long\def\spp{\operatorname{supp}}
\global\long\def\spo{\operatorname{supp}_{o}}
\global\long\def\del{\partial}
\global\long\def\lsc{\operatorname{Lsc}}
\global\long\def\GU{\CG^{(0)}}
\global\long\def\HU{\CH^{(0)}}
\global\long\def\AU{\CA^{(0)}}
\global\long\def\BU{\CB^{(0)}}
\global\long\def\CUU{\CC^{(0)}}
\global\long\def\DU{\CD^{(0)}}
\global\long\def\QU{\CQ^{(0)}}
\global\long\def\TU{\CT^{(0)}}
\global\long\def\CUUU{\CC'{}^{(0)}}
\global\long\def\dom{\operatorname{dom}}
\global\long\def\ran{\operatorname{ran}}
\global\long\def\AUl{(\CA^{l})^{(0)}}
\global\long\def\BUl{(B^{l})^{(0)}}
\global\long\def\HUp{(\CH^{p})^{(0)}}
\global\long\def\sym{\operatorname{Sym}}
\global\long\def\stab{\operatorname{Stab}}
\newcommand{\cat}[0]{\operatorname{CAT}(0)}
\global\long\def\properlength{proper}
\global\long\def\deg{\operatorname{deg}}
\global\long\def\isom{\operatorname{Isom}}
\global\long\def\interior#1{#1^{\operatorname{o}}}
	\global\long\def\ln{\operatorname{ln}}

\author{Xin Ma}
\email{xma1@memphis.edu}
\address{Department of Mathematics,
	University of Memphis,
	Memphis, TN, 38152}


\date{Feb 27, 2022}

\author{Daxun Wang}
\email{daxunwan@buffalo.edu}
	\address{Department of Mathematics,
		Sate University of New York at Buffalo,
		Buffalo, NY, 14260}
\begin{abstract}
	In this paper, we study boundary actions of $\cat$ spaces from a point of view of topological dynamics and $C^*$-algebras. First, we investigate the actions of right-angled Coexter groups and right-angled Artin groups with finite defining graphs on the visual boundaries and the Nevo-Sageev boundaries of their natural assigned $\cat$ cube complexes. In particular, we establish (strongly) pure infiniteness results for reduced crossed product $C^*$-algebras of these actions through investigating the corresponding $\cat$ cube complexes and establishing necessary dynamical properties such as minimality, topological freeness and pure infiniteness of the actions. In addition, we study actions of fundamental groups of graphs of groups on the visual boundaries of their Bass-Serre trees. We show that the existence of repeatable paths essentially implies that the action is $2$-filling, from which, we also obtain a large class of unital Kirchberg algebras. Furthermore, our result also provides a new method in identifying $C^*$-simple generalized Baumslag-Solitar groups.  The examples of groups obtained from our method have $n$-paradoxical towers in the sense of \cite{G-G-K-N}. This class particularly contains non-degenerated free products, Baumslag-Solitar groups and fundamental groups of $n$-circles or wedge sums of $n$-circles. 
\end{abstract}
\maketitle

\section{Introduction}
Boundaries of certain $\cat$ spaces and group actions on them play important roles in the study of groups, geometry and topology. Motivating examples include the Gromov boundaries of hyperbolic spaces as well as hyperbolic groups acting on their Gromov boundaries. For $\cat$ spaces beyond the hyperbolic world, there are many boundaries with similar flavor that one may consider, as we will list below. Suppose a group $G$ acts on a $\cat$ space $X$ by isometry. Then under some natural assumptions, there is an induced topological action of $G$ on these boundaries. Dynamical properties on the boundaries have been proved to play a significant role in investigating many useful properties of the acting groups or the spaces themselves such as the Tits alternative,  Yu's Property A, a-T-amenability, and thus leads to many other applications in topology.

On the other hand, as one of our main motivation in the current paper, reduced crossed products of the form $C(X)\rtimes_r G$ arising from topological dynamical systems, say from $(X, G,
\alpha)$ for a countable discrete group $G$, a locally compact Hausdorff space $X$ and a continuous action
$\alpha$, have long been an important source of examples and motivation for the study of $C^\ast$-algebras. Our one goal in this paper is to continue the study of the first author in \cite{M2} and \cite{M} to investigate the pure infiniteness of reduced crossed product $C^*$-algebras. See also \cite{A-D}, \cite{L-S}, \cite{J-R}, \cite{R-S} and a very recent progress \cite{G-G-K-N} for more information in this direction.

\textit{Pure infiniteness} of a $C^*$-algebra, reflecting a kind of paradoxical nature, is an important regularity property of $C^*$-algebras. It has many characterizations (see \cite{K-Rord}, \cite{Kir-Rord} and \cite{P-R}) and plays an essential role in the celebrated classification theorem by Kirchberg and Phillips (see. e.g., \cite{Phillips}). On the other hand, beyond the classification theorem, the property of pure infiniteness, and its variant stronge pure infiniteness have their own interest to be studied as well.

Therefore, in this paper, we study boundary actions of certain $\cat$ spaces from a topological dynamical and operator algebraic viewpoints to determine when the reduced crossed product $C^*$-algebras of the boundary actions are purely infinite. This study will yield new and interesting examples belonging to the class of strongly purely infinite $C^*$-algebras. 

Our first motivating examples are  actions of certain non-amenable groups on the visual boundaries that have a strong paradoxical flavor. For example, as a generalization of the hyperbolic case, if a group $G$ acts on a proper $\cat$ space $X$ by isometry in a non-elementary way (see \cite{Ham}), then any rank-one element $g$ in $G$ performs the classical North-South dynamics on the visual boundary $\partial X$(see \cite{Ham}), i.e., there exist \textit{attracting} and \textit{repelling} fixed points of $g$ such that the positive powers of $g$ contract the whole boundary except the repelling fixed point into the attracting points. This strong contracting dynamics implies that the action of $G$ on $\partial X$ has dynamical comparison in the sense of \cite[Definition 3.2]{D} and has no $G$-invariant measures. This condition of comparison is also equivalent to the so-called pure infiniteness of the action in the sense of \cite[Definition 3.5]{M}. Moreover, it was shown in \cite{Ham} that such an action is also minimal. Then under the assumption that the action is topologically free, its reduced crossed product is simple and purely infinite. See e.g. \cite{M2}. 

Enlarging our scope, observe that many examples of purely infinite reduced crossed product appeared in the literature arise from boundary actions that have similar strong contracting dynamics. This implies that  it is worth investigating boundary actions of $\cat$ spaces in a more systematical way. The first step is to determine which boundary one should look at because there are a lot of candidates beyond the hyperbolic world. We enumerate several here and warn that this is not a complete list at all.  Let $X$ be a $\cat$ space, one may consider
\begin{enumerate}
	\item the visual boundary $\partial X$,
	\item the horofunction boundary,
	\item Contracting boundary or equivalently in the $\cat$ case, the Morse boundary (see \cite{CS1}, \cite{Cor} and \cite{C-M}).
	\item $\kappa$-Morse boundary (see \cite{Q-R}).
\end{enumerate} 
In this list, the visual boundary might be the most ``transparent'' compact Hausdorff boundary associated to a $\cat$ space. Similarly to the Gromov boundary, it contains  the equivalence class of geodesics that are almost in the same direction. The horofunction boundary are equivalent to the visual boundary. See \cite{B-H} for more information. However, the visual boundary is not a quasi-isometric invariant and that is one of the motivation why the boundaries in (3) and (4) above are invented. However, the topologies on contracting and ($\kappa$-)Morse boundaries are no longer compact if the space is not hyperbolic. See \cite[Theorem 10.1]{C-M}, and \cite[Propsoition 6.6]{Q-R} and \cite[Theorem 1.1]{C-C-S}. It also seems unknown whether there exists a non-hyperbolic $\cat$ space with locally compact contracting or $\kappa$-Morse boundary. Therefore, boundaries in (3) and (4) are out of our scope at this moment because topological spaces considered for $C^*$-algebras are usually assumed to be locally compact Hausdorff.

Nevertheless, If we consider additional structures on the $\cat$ spaces, e.g.  cube complexes, we have more boundaries at hands which are of combinatorial flavors,
\begin{enumerate}
	\item[(5)] Roller boundary $\CR(X)$ (see \cite{Roller}) and
	\item[(6)] Nevo-Sageev boundary $B(X)$ as a subset of $\CR(X)$ (see \cite{N-S}).
\end{enumerate}

We remark that the Roller boundary can be identified with the horofunction boundary of the $1$-skeleton $X^1$ of the complex $X$ with $\ell_1$-metric. Based on these discussions, we mainly consider the visual boundary, the Roller boundary and the Nevo-Sageev boundary in this paper and consider the actions of right-angled Coxeter groups (RACGs) and right-angled Artin groups (RAAGs) as well as the actions of fundamental groups of graph of groups on the visual boundaries of their Bass-Serre trees in this paper.

Our first contribution in this paper is to detect the structure of the reduced crossed products of actions of RACGs and RAAGs on the boundaries above based only on the information of their defining graphs. The main outcome are (strongly) pure infiniteness results for the reduced crossed products. To establish this, our main method is to verify necessary dynamical properties such as minimality, topological freeness and  pure infiniteness of the actions (see Theorem \ref{thm: pure inf main} below) as well as using the algebraic structure of tensor products. Note that the above dynamical conditions are in general not easy to establish. Nevertheless, actually as our another motivation to consider boundary actions, in some cases, our actions are boundary actions in the sense of Furstenberg and thus are topologically free provided that the acting groups are $C^*$-simple and  there exists points with amenable stabilizers (see \cite{B-K-K-O}).

As a result, we have the following main theorems on $\cat$ cube complexes. The involved notions will be introduced in next sections. Let $\Gamma=(V, E)$ be a finite simple graph. For simplicity, we denote by $G_\Gamma$ the RACG $W_\Gamma$ or the RAAG $A_\Gamma$ and $X_\Gamma$ the corresponding Davis complex $\Sigma_\Gamma$ or the universal cover of the Salvetti complex $ \tilde{S_\Gamma}$, respectively. There is a canonical way to write all $G_\Gamma$ into direct products of special subgroups by using joins of the defining graph $\Gamma$ with a form $G_\Gamma=G_{\Gamma_1}\times\dots G_{\Gamma_m}$. Similarly, using the joins of $\Gamma$, one can also decompose $X_\Gamma$ to be
$X_\Gamma=X_{\Gamma_1}\times\dots\times X_{\Gamma_m}$ in a corresponding way.
 \begin{customthm}{A}(Theorem \ref{thm: main two})
		Let $G_\Gamma\curvearrowright X_\Gamma$ where $X_\Gamma$ is essential and has at least one non-Euclidean irreducible factor $X_{\Gamma_i}$ in the decomposition above.  Suppose 
		\begin{enumerate}
			\item  $G_\Gamma=W_\Gamma$ has no special subgroup $D_\infty$ in the  RACG case;
			\item  $G_\Gamma=A_\Gamma$ has no special subgroup $\Z$  in the RAAG case.
		\end{enumerate}
		Then the reduced crossed product $A=C(B(X_\Gamma))\rtimes_r G_\Gamma$ of the induced action on the Nevo-Sageev boundary $\beta: G_\Gamma\curvearrowright B(X_\Gamma)$ is unital simple separable and purely infinite. In addition, in the RACG case, $A$ is nuclear as well and thus a Kirchberg algebra satisfying the UCT.
	\end{customthm}
We remark that the absence of a special subgroup $D_\infty$ or $\Z$ for $G_\Gamma$ is equivalent to that there is no factor of $D_\infty$ or $\Z$, which is called Euclidean factors, in the direct product decomposition $G_\Gamma=G_{\Gamma_1}\times\dots G_{\Gamma_m}$ above. On the other hand, such a Euclidean factor $D_\infty$ or $\Z$, has its own action on its own Nevo-Sageev boundary, which is exactly a set consisting two points, denoted by $\{\breve{0}, \breve{1}\}$ in this paper. Moreover, in the RACG case, $D_\infty$ acts on the set by alternating while in the RAAG case, $\Z$ acts on it trivially. See more in Section 4.

Now suppose $G_\Gamma$ has such Euclidean factors. One can still verify pure infiniteness in the sense of Definition \ref{defn: comparison-pure infiniteness}  of the action $\beta$ on $B(X_\Gamma)$ by using Proposition \ref{prop: p.i. in product with trivial action}. To observe more, in the RACG case, $\beta$ is still minimal and in the RAAG case, $\beta$ has finitely many closed invariant sets (see Section 4). However, it follows from Proposition \ref{prop: p.i. in product with trivial action} that $\beta$ is no longer topologically free and in particular is not essentially free so that Theorem \ref{thm: pure inf main} cannot be applied here to establish the (strongly) pure infiniteness of the crossed product.  From the viewpoint of $C^*$-algebras, after we add the Euclidean factors $D_\infty$ or $\Z$ in the group, the Nevo-Sageev boundary increases in a way that the crossed product do not have a good ideal structure any more. Nevertheless, because the Nevo-Sageev boundary preserves the product structure, we still know what does $C(B(X_\Gamma))\rtimes_rG_\Gamma$ look like and are able to show they are strongly purely infinite. First it follows from Theorem \ref{thm: decompostion} that any $G_\Gamma$ can be decomposed into a direct product $G_\Gamma=G_{\Gamma'}\times H^n$, for some $n\in N$, where $\Gamma'$ is a subgraph of $\Gamma$ and $H$ is the Euclidean factor, i.e., $H=D_\infty$ if $G=W_\Gamma$ and $H=\Z$ if $G=A_\Gamma$. 

\begin{cor}(Corollary \ref{cor: full result on B(X)})
Let $G_\Gamma=G_{\Gamma'}\times H^n$ be the decomposition mentioned above. Write $A=C(\partial X_\Gamma)\rtimes_r G_\Gamma$. Then one has
\begin{enumerate}
\item In the RACG case,  one has $A=(C(\partial X_{\Gamma'})\rtimes_r G_{\Gamma'})\otimes \bigotimes_{i=1}^n (C(\{\breve{0}, \breve{1}\})\rtimes_r D_\infty)$, where the action of $D_\infty$ on $\{\breve{0}, \breve{1}\}$ is by alternating and $C(X_{\Gamma'})\rtimes_r G_{\Gamma'}$ is unital simple separable purely infinite.
\item In the RAAG case, one has $A=(C(\partial X_{\Gamma'})\rtimes_r G_{\Gamma'})\otimes C(\{\breve{0}, \breve{1}\}^n)\otimes C(\T^n)$ in which $\T$ is the unit circle and $C(X_{\Gamma'})\rtimes_r G_{\Gamma'}$ is unital simple separable purely infinite.
\end{enumerate}
In the RACG case, $A$ is $\CO_\infty$-stable and actually, in either case, $A$ is strongly purely infinite.
\end{cor}

For the visual boundaries of the cube complexes $X_\Gamma=\Sigma_\Gamma$ or $\tilde{S_\Gamma}$, in the irreducible case, by establishing the necessary topological dynamical properties mentioned above, we have the following result.

\begin{customthm}{B}(Theorem \ref{cor: main3})
Let $\Gamma=(V, E)$ be a finite simple graph without joins. 
\begin{enumerate}
	\item if $|V|\geq 3$ then $C(\partial \Sigma_\Gamma)\rtimes_r W_\Gamma$ is simple and  purely infinite.
	\item If $|V|\geq 2$ then $C(\partial \tilde{S_\Gamma})\rtimes_r A_\Gamma$ is simple and purely infinite.
\end{enumerate}
\end{customthm}

Another type of actions of  groups on the visual boundary considered in this paper is the fundamental group of a graph of groups acting on its  Bass-Serre tree. Follow the notations in \cite{B-M-P-S-T} for graph of groups $\CG=(\Gamma, G)$ and  denote by  $\pi_1(\CG, v)$ the fundamental group of $\CG$ at a base vertex $v$ and $v\partial X_{\CG}$ the boundary of the Bass-Serre tree, we have the following other main theorems.  We remark that the following assumptions in Theorem C and D on graphs of groups $\CG$ are very mild so that one can easily construct examples satisfying the theorems, e.g., examples in Theorem E below.  One may also want to compare Theorems C and D to the pure infiniteness result of $C^*$-algebras obtained in \cite{B-M-P-S-T} originally. The main tool there is the local contraction of the action introduced in \cite{A-D} and \cite{L-S}. The relation between the local contraction and our pure infiniteness of actions were discussed in \cite[Theorem 5.8]{M}. On the other hand, it follows from \cite[Theorem 6.15]{M} that there exists purely infinite actions of $\Z_2*\Z_3$ on $\{0, 1\}^\N\times \R$, which is not locally contracting. For the other notions appeared in the following theorems, we also refer to Section 5 for the definitions. 

\begin{customthm}{C}(Theorem \ref{thm: boundary action})
Let $\Gamma=(V, E)$ be a locally finite non-singular graph and $\CG=(\Gamma, G)$ a graph of groups. Suppose 
\begin{enumerate}
	\item $v\partial X_{\CG}$ is infinite;
	\item $\xi$ can flow to $e$ for any $\xi\in \partial X_{\CG}$ and $e\in E$; and
	\item there is a repeatable path $\mu=g_1e_1\dots g_ne_n$ with $|\Sigma_{\bar{e}_n}|\geq 2$.
\end{enumerate}
Then the natural action $\beta: \pi(\CG, v)\curvearrowright v\partial X_{\CG}$ is a strong boundary action. In particular, $\beta$ is a $\pi_1(\CG, v)$-boundary action. If, in addition, each $G_e$ is amenable and $\pi_1(\CG, v)$ is $C^*$-simple, then the action $\beta$ is topologically free and thus the crossed product $C(v\partial X_{\CG})\rtimes_r \pi_1(\CG, v)$ is a unital Kirchberg algebra satisfying the UCT.
\end{customthm}

In particular, if we restrict to Generalized Baumslag-Solitar (GBS) groups/graphs, we have the following theorem. Note that this result also provides a new method in identifying $C^*$-simple GBS groups.

\begin{customthm}{D}(Theorem \ref{thm: main one})
Let $\CG=(\Gamma, G)$ be a locally finite non-singular GBS graph of groups in which $\Gamma=(V, E)$ is a finite graph. 
Suppose 
\begin{enumerate}
	\item $v\partial X_{\CG}$ is infinite;
	\item $\xi$ can flow to $e$ for any $\xi\in \partial X_{\CG}$ and $e\in E$;
	\item there is a repeatable path $\mu=g_1e_1\dots g_ne_n$ with $|\Sigma_{\bar{e}_n}|\geq 2$; and
	\item $\CG$ is not unimodular.
\end{enumerate}
Then the natural action $\beta: \pi_1(\CG, v)\curvearrowright v\partial X_{\CG}$ is an amenable topologically free strong boundary action and the crossed product $C(v\partial X_{\CG})\rtimes_r \pi_1(\CG, v)$ is a unital Kirchberg algebra satisfying the UCT. Furthermore, $\pi_1(\CG, v)$ is $C^*$-simple.
\end{customthm}

Denote by $\CC$ the class of all  fundamental groups of graph of groups satisfying Theorems C and D, which includes the following specific examples. We remark all groups in $\CC$ have $2$-paradoxical towers in the sense of \cite{G-G-K-N}. Then any minimal topologically free amenable actions of these groups on a compact metrizable space yield a unital Kirchberg algebra satisfying the UCT. See more in Theorem \ref{thm: final}. To detect members in $\CC$, we have the following examples.

\begin{customthm}{E}(Theorem \ref{eg: main4})
	Still write $\CG=(\Gamma, G)$, the class $\CC$ particularly contains the following groups.
	\begin{enumerate}
		\item $C^*$-simple $\pi_1(\CG, v)$ in which  $\CG$ satisfies assumptions (1)-(3) of Theroem C and each $G_e$ is amenable. This includes Example \ref{eg: new1} In particular, this includes $G*F$ such that $(|G|-1)(|F|-1)\geq 2$.
		\item $C^*$-simple GBS groups $\pi_1(\CG, v)$ appeared in Theorem D.  This includes non-degenerated $BS(k, l)$ where $(|k|-1)(|l|-1)\geq 2$ in Example \ref{eg: new2} and certain GBS groups of $n$-circles in Example \ref{eg: new3} as well as some GBS groups of wedge sum of $n$-circles for $m$ times in Example \ref{eg: new4}. In addition, if $n\geq 2$ or $m\geq 2$, these are not non-degenerated BS groups by Remark \ref{rmk: final}.
	\end{enumerate} 
\end{customthm}

\begin{rmk}
	\begin{enumerate}
	\item During the preparation of the current paper. Gardella, Geffen, Kranz and Naryshkin posted on arXiv a paper \cite{G-G-K-N} on similar topics. We remark that our Theorem A(1) has generalized their result in \cite[Example 4.8]{G-G-K-N}. On the other hand, once the minimality and the topological freeness has been proved through our method (see Section 4), by combining the amenability of action on $\CR(X)$ obtained in \cite{Lec}, one can also apply their results to obtain a different proof of Theorem A(1). However, the other theorems in this paper cannot be established in this way, because to the best knowledge of the authors, it is unknown whether those actions are amenable. See more in Remark \ref{rmk: aukward}.
	\item In the same day that the authors of  this paper submit the first version of the current paper to arXiv,  there is a new version of \cite{G-G-K-N}  appeared on arXiv in which a different approach is used to show non-degenerated free products and Baumslag-Solitar groups as new examples of groups with $n$-paradoxical towers. These examples are covered by our Theorem E as well. Furthermore, in the current second version, we added more examples other than BS groups obtained in the first version. 
	\end{enumerate}

\end{rmk}

Our paper is organized in the following way. In Section 2, we review some necessary concepts, definitions and preliminary results. In Section 3, we establish all topological dynamical theorems and link them to pure infiniteness of reduced crossed products for use later. In Section 4, we focus on the action of RACGs and RAAGs on the  Roller boundaries and the Nevo-Sageev boundaries to establish all necessary dynamical properties described in Section 3 and thow the reduced crossed product is (strongly) purely infinite. In Section 5, we investigate actions of irreducible RACGs and RAAGs on visual boundaries of their complexes as well as the fundamental groups of graphs of groups on the visual boundaries of their Bass-Serre trees to study pure infiniteness of related reduced crossed products, $C^*$-simlicity of certain GBS groups and the groups with $n$-paradoxical towers.

\section{Preliminaries}
In this section, we recall some terminologies and definitions used in the paper.
\subsection{Groups, topological dynamical systems and their $C^*$-algebras }
Let $G$ be a countable discrete group, $X$ a locally compact Hausdorff space and $\alpha: G\curvearrowright X$ denotes a continuous action of $G$ on $X$.  We write $M_G(X)$ for the set of all $G$-invariant regular Borel probability measures on $X$.

We say an action $\alpha: G\curvearrowright X$ is \textit{minimal} if all orbits are dense in $X$. Recall that an action $\alpha: G\curvearrowright X$ is said to be \emph{essentially free} provided that, for every closed $G$-invariant subset $Y\subset X$, the subset of points in $Y$ with the trivial stabilizer, say $\{x\in Y: \stab_G(x)=\{e\}\}$, is dense in $Y$, where $\stab_G(x)=\{t\in G: tx=x\}$. An action is said to be \textit{topologically free} provided that the set $\{x\in X: \stab_G(x)=\{e\}\}$, is dense in $X$ and this is equivalent to that the fixed point set $\{x\in X: tx=x\}$ of each nontrivial element $t$ of $G$ is nowhere dense. It is not hard to see that essentially freeness means that the restricted action to each $G$-invariant closed subspace  is topologically free with respect to the relative topology and thus these two concepts are equivalent when the action is minimal. We refer to \cite{B-O} for standard construction of (reduced) crossed product $C^*$-algebras $C_0(X)\rtimes_r G$ for topological dynamical systems.

In the case that $X$ is compact, It is well known that if the action $G\curvearrowright X$ is topologically free and minimal then the reduced crossed product $C(X)\rtimes_r G$ is simple (see \cite{A-S}) and it is also known that the crossed product $C(X)\rtimes_r G$ is nuclear if and only if the action $G\curvearrowright X$ is amenable (see \cite{B-O}). Archbold and Spielberg \cite{A-S} showed that $C(X)\rtimes G$ is simple if and only if the action is minimal, topologically free and \textit{regular} (meaning that the reduced crossed product coincides with the full crossed product). These imply that $C(X)\rtimes_r G$ is simple and nuclear if and only if the action is minimal, topologically free and amenable. 

A type of topological dynamical systems of the particular interest are $G$-boundary actions. Now, let $X$ be compact and denote by $P(X)$ the set of all probability measures on $X$. Furstenberg provided the following definition in \cite{F}.

\begin{defn}\label{defn: boundary}
	\begin{enumerate}
		\item A $G$-action $\alpha$ on $X$ is called \textit{strongly proximal} if for any probability measure $\eta\in P(X)$, the closure of the orbit $\{g\eta: g\in G\}$ contains a Dirac mass $\delta_x$ for some $x\in X$ 
		\item A $G$-action $\alpha$ on a compact Hausdorff space $X$ is called a $G$-\textit{boundary action} if $\alpha$ is minimal and strongly proximal.
	\end{enumerate}
\end{defn}
Topological freeness of a $G$-boundary action is linked to $C^*$-simplicity of $G$, i.e., $C^*_r(G)$ is simple. We refer to \cite{Harpe}, \cite{B-K-K-O} and \cite{K-K} for this topic.

Let $A$ be a unital $C^*$-algebra. A non-zero positive element $a$ in $A$ is said to be\textit{ properly infinite} if
$a\oplus a\precsim a$, where $\precsim$ is the Cuntz subequivalence relation, for which we refer to \cite{A-P-T} as a standard reference. A $C^\ast$-algebra $A$ is said to be \textit{purely infinite} if there are no characters
on $A$ and if, for every pair of positive elements $a,b\in A$ such that $b$ belongs to the closed ideal in $A$ generated by $a$, one has $b\precsim a$. It was proved in \cite{K-Rord} that a $C^\ast$-algebra $A$ is purely infinite if and
only if every non-zero positive element $a$ in $A$ is properly infinite. In addition, in \cite[Definition 5.1]{Kir-Rord}, Kirchberg and R{\o}rdam also introduced a stronger version of pure infiniteness for $C^*$-algebras called \textit{strongly pure infiniteness}. Denote by $\CO_\infty$ the Cuntz algebra of infinite generators. The condition of $\CO_\infty$-stability for $A$, i.e. $A\otimes \CO_\infty\simeq A$,  implies that $A$ is strongly purely infinite, and if $A$ is separable and nuclear, these two conditions are equivalent. See by \cite[Theorem 8.6]{Kir-Rord}.

In this paper, we will address on \textit{right-angled Coexter groups} and \textit{right-angled Artin groups}, abbreviated by RACGs and RAAGs, respectively. We recall their definitions by using \textit{defining graphs}, which are finite simple graph $\Gamma=(V, E)$, i.e.,  the vertex set $V$ is finite.

\begin{defn}
	For a finite simple graph $\Gamma=(V, E)$, the corresponding RACG $W_\Gamma$ is defined to be
	\[W_\Gamma=\la V: v^2_i=e\ \text{for any }1\leq i\leq n\ \text{and } v_iv_j=v_jv_i\ \text{for any }(v_i, v_j)\in E\ra.\]
	The corresponding RAAG $A_\Gamma$ is defined to be 
	\[A_\Gamma=\la V: v_iv_j=v_jv_i\ \text{for any }(v_i, v_j)\in E\ra.\]
\end{defn}

Let $\Gamma=(V, E)$ be a finite simple graph. Let $\Lambda$ be a subgraph of $\Gamma$. Then the corresponding subgroup $W_\Lambda$ (resp. $A_\Lambda$) of $W_\Gamma$ (resp. $A_\Gamma$)  is called a \textit{special subgroup}. We say $\Gamma$ is a \textit{join} if there are two subgraphs $\Gamma_1=(V_1, E_1)$ and $\Gamma_2=(V_2, E_2)$ of $\Gamma$ such that $V_1$ and $V_2$ form a partition of $V$ and every vertex in $V_1$ is adjacent to every vertex in $V_2$. In this situation, we write $\Gamma=\Gamma_1\star\Gamma_2$. If there is no join for $\Gamma$, we call the corresponding RACG $W_\Gamma$ and RAAG $A_\Gamma$ \textit{irreducible}. In addition, for each RACG and RAAG,  one naturally assign it with a $\cat$ cube complex constructed from its Cayley graph, which is called \textit{Davis complex} and the \textit{universal cover of the Salvetti complex}, respectively so that the RACG and the RAAG act on them by isometry cocompactly, respectively. We will leave the definitions and discussions of these two specific complexes until Section 4. Instead, we recall some general facts on $\cat$ cube complexes here. 

\subsection{$\cat$ cube complexes and their boundaries}
We refer to \cite{N-S}, \cite{C-S}, \cite{Le} and \cite{Char} for general information of $\cat$ cube complexes.

\begin{defn}
A $\cat$ \textit{cube complex} is a simply connected cell complex whose cells are Euclidean cubes $[0, 1]^d$ of various dimensions. In addition, the link of each $0$-cell, i.e., vertex, is a \textit{flag} complex, which is a simplicial complex such that any $n+1$ adjacent vertices belong to an $n$-simplex.
\end{defn}

We say a $\cat$ cube complex $X$ \textit{finite dimensional} if there is a uniform upper bound on the dimension of cubes in $X$. Let $X$ be a $\cat$ cube complex. A \textit{midcube} of a cube $[0, 1]^d$, is the restriction of a coordinate of the cube to be $1/2$. A \textit{hyperplane} $\hat{h}$ is a connected subspace of $X$ with the property that for each cube $C$ in $X$, the intersection $\hat{h}\cap C$ is either a midcube of $C$ or empty. Let $e$ be an edge in $X^1$, we say a hyperplane $\hat{h}$ is dual to $e$ if $\hat{h}\cap e\neq \emptyset$.  In general, $\hat{h}$ separates $X$ into precisely two components, called \textit{halfspaces}, denoted by $h$ and $h^*$. $X$ is said to be \textit{essential} if given any half space $h$ in $X$, there is a vertex in $h$ arbitrary far from $\hat{h}$. Similarly, we say a group $G\leq \aut(X)$ acts \textit{essentially} on $X$ if no $G$-orbit remains in a bounded neighborhood of a halfspace of $X$. Here $\aut(X)$ is the automorphism group of $X$ consisting all isometries that preserve the cubical structures. A $\cat$ cube complex $X$ is said to be \textit{cocompact} if the action on $X$ of the group $\isom(X)$ consisting all isometries of $X$  is cocompact.

A $\cat$ cube complex $X$ is said to be \textit{irreducible} if it cannot be written as a product of two $\cat$ cube complexes. Otherwise, we say $X$ is reducible. Let $n\in \N$. An $n$-dimensional \textit{flat} is an isometrically embedded copy of $n$-dimensional Euclidean space $\mathbb{E}^n$ (in the usual $\cat$ metric). A unbounded cocompact $\cat$ cube complex $X$ is said to be \textit{Euclidean} if $X$ contains a $\aut(X)$-invariant flat. Otherwise, we say $X$ is non-Euclidean. An unbounded essential $\cat$ cube complex whose irreducible factors are all non-Euclidean is called a \textit{strictly non-Euclidean} complex. 

We also consider the $1$-skeleton of a $\cat$ cube complex, which usually equipped with the usual $\ell_1$-metric (called the path metric or the combinatorial metric as well). For the finite-dimensional case, the $\ell_1$-metric and the usual $\cat$ metric on $X$  are quasi-isometric to each other.

\begin{lem}\cite[Lemma 2.2]{C-S}\label{lem: quasi-isometric}
Let $(X,d)$ be a finite-dimensional $\cat$ cube complex, where $d$ is the usual $\cat$ metric. Then $(X,d)$ is quasi-isometric to its $1$-skeleton endowed with the combinatorial metric.
\end{lem}

Finally, we recall that one may assign several compact Hausdorff boundaries to a $\cat$ cube complex $X$. We refer to \cite[Section 1.3]{N-S} and \cite{B-H} for more detailed information. If a group $G$ acting on $X$ by isometry, sometimes, the action can be naturally extended to the boundary as a topological action, which will yield interesting topological dynamical systems and $C^*$-algebras. In this paper, we mainly care about the \textit{visual boundary} (see. e.g. \cite{B-H}) and the \textit{Nevo-Sageev boundary} introduced in \cite{N-S}.  We still leave the definitions to Section 4 and 5.

\section{Comparison properties and pure infiniteness of dynamical systems}
In this section, we recall several useful dynamical notions appeared in the literature that have a paradoxical flavor and in fact imply the reduced crossed products are purely infinite. We also provide some new criteria for these notions to hold, which will be applied in the following sections. Let $G$ be a countable discrete group and $X$ a locally compact Hausdorff space. Let $G\curvearrowright X$ be a continuous action. The following definition appeared in \cite{L-S}.  See also \cite{Gla2}.

\begin{defn}\cite[Definition 1]{L-S}\label{defn: strong boundary}
Let $X$ be a compact Hausdorff space. We say an action $G\curvearrowright X$ is a \textit{strong boundary action} (or \textit{extreme proximal}) if for any compact set $F$ and  non-empty open set $O$ there is a $g\in G$ such that $gF\subset O$.  
\end{defn}

We remark that $G\curvearrowright X$ is a strong boundary action, then in\cite{Gla1}, Glasner showed that $X$ is a $G$-boundary in the sense of Definition \ref{defn: boundary}. On the other hand, it was proved in \cite{L-S} that the reduced crossed product of a topological free strong boundary action is simple and purely infinite. Then the notion of strong boundary action has been generalized in \cite{J-R} to $n$-filling actions.

\begin{defn}\cite{J-R}\label{defn: n-filling}
	An action $\alpha: G\curvearrowright X$  on a compact Hausdorff space $X$ is said to be $n$-\textit{filling} if for any non-empty open sets $O_1, \dots, O_n$ there are $n$ group elements $g_1,\dots, g_n\in G$ such that $\bigcup_{i=1}^ng_iO_i=X$. 
\end{defn}

It is not hard to see strong boundary actions are exactly the $2$-filling actions and it was proved in \cite{J-R} that reduced crossed products of topologically free $n$-filling actions are also simple and purely infinite. Note that all $n$-filling actions are necessarily minimal.  Then, in \cite{M2}, the first author observed that the \textit{dynamical comparison}, first introduced by Winter and then refined by Kerr in \cite{D}, also serves as an generalization of the $n$-filling property and still implies the pure infiniteness of the reduced crossed products in the case that $\alpha$ is minimal and there is no $G$-invariant probability measure on $X$. To move furthermore, in \cite{M2} and \cite{M}, under the assumption that there is no invariant measures, which is usually necessary for a reduced crossed product to be purely infinite, the theory surrounding dynamical comparison property actually has been established in a more general setting of non-minimal actions and even for locally compact Hausdorff \'{e}tale groupoids. In the current paper, we only deal with the transformation groupoids case, i.e., countable discrete group $G$ acting on locally compact Hausdorff spaces $X$.  In addition, throughout the paper, we write $A\sqcup B$ to indicate that the union of sets $A$ and $B$ is a disjoint union and denote by $\bigsqcup_{i\in I}A_i$ the disjoint union of the family $\{A_i: i\in I\}$.

\begin{defn}(\cite{D}, \cite{M})
	Let $G\curvearrowright X$ be a continuous action on a locally compact Hausdorff space $X$. Let $O, V$ be non-empty open sets in $X$ and $F$ a compact set in $X$. 
	\begin{enumerate}[label=(\roman*)]
		\item  We write $F\prec O$ if there is an open cover $\{U_1, \dots, U_n\}$ of $F$ and group elements $g_1,\dots, g_n\in G$ such that $\{g_1U_1,\dots, g_nU_n\}$ is a disjoint family of open sets contained in $O$, i.e., $\bigsqcup_{i=1}^ng_iU_i\subset O$. 
		\item We say $V$ is \textit{dynamical subequivalent} to $O$, denoted by $V\prec O$, if $F\prec O$ for any compact set $F\subset V$. 
		\item We say $V$ is \textit{paradoxical subequivalent} to $O$, denoted by $V\prec_2 O$, if $F\prec_2 O$ for any compact set $F\subset V$ in the sense that there are disjoint non-empty open sets $O_1, O_2\subset O$ such that $F\prec O_1$ and $F\prec O_2$. 
	\end{enumerate}
\end{defn}

The following concepts were introduced in \cite{D}, \cite{M2} and \cite{M}.
\begin{defn}\label{defn: comparison-pure infiniteness}
	 Let $\alpha: G\curvearrowright X$ be a continuous action on a locally compact Hausdorff space $X$.
	 \begin{enumerate}[label=(\roman*)]
	 \item $\alpha$ is said to have dynamical comparison if $U\prec V$ whenever $\mu(U)<\mu(V)$ for any $\mu\in M_G(X)$.
	 \item $\alpha$ is said to have \textit{paradoxical comparison}  if $O\prec_2 O$ for any non-empty open set $O$ in $X$.
	 \item $\alpha$ is said to be \textit{purely infinite} if $U\prec_2 V$ for any  non-empty open sets $U, V$ satisfying $U\subset G\cdot V$.
	 \item $\alpha$ is said to be \textit{weakly purely infinite} if $U\prec V$ for any  non-empty open sets $U, V$ satisfying $U\subset G\cdot V$.
	 \end{enumerate}
\end{defn}

It has been observed in \cite{M2} that all $n$-filling actions, thus including all strong boundary actions, satisfy dynamical comparison and have no invariant probability measures.  For the relation among notions above, the first author proved the following theorem in \cite{M}, which was written in the language of groupoids.
\begin{thm}\cite[Theorem 5.1]{M}\label{thm: pure infiniteness of action}
		Let $\alpha: G\curvearrowright X$. Consider the following conditions.
	\begin{enumerate}[label=(\roman*)]

		\item $\alpha$ has dynamical comparison and $M_G(X)=\emptyset$.
		
		\item $\alpha$ is purely infinite.
		
		\item $\alpha$ has paradoxical comparison.
		
		\item $\alpha$ is weakly purely infinite.
	\end{enumerate}
	Then (i)$\Rightarrow$(ii)$\Leftrightarrow$(iii)$\Rightarrow$(iv). If $\alpha$ is minimal then they are equivalent.
\end{thm}

The following result was essentially established in \cite{M2} as our main tool for application in the next sections. We remark that a version of locally compact Hausdorff \'{e}tale groupoids of the following results have been  established in \cite{M}. 

\begin{theorem}\cite[Theorem 1.1 Corollary 1.4]{M2}\label{thm: pure inf main}
	Let $G$ be a countable discrete infinite group, $X$ a compact Hausdorff space and $\alpha: G\curvearrowright X$ an action of $G$ on $X$.  Suppose $\alpha$ is purely infinite. If either
	\begin{enumerate}
		\item $\alpha$ is minimal and topologically free,  or
		\item $G$ is exact and $\alpha$ is essentially free as well as there are only finitely many $G$-invariant closed sets in $X$,
	\end{enumerate}
then the reduced crossed product $C(X)\rtimes_r G$ is strongly purely infinite. In the first case $C(X)\rtimes_r G$ is simple. In the second case $C(X)\rtimes_r G$ has finitely many ideals.
\end{theorem}
\begin{proof}
Under the assumptions, it has been proved that $C(X)\rtimes_r G$ is purely infinite in \cite{M2}. If $\alpha$ is minimal and topologically free then $C(X)\rtimes_r G$  is simple and thus strongly purely infinite. In the second case,  first it was proved in \cite[Theorem 1.20]{Si} that $C(X)$ separates the ideals of $A=C(X)\rtimes_rG$. Then since there are only finitely many $G$-invariant closed ideals in $C(X)$ by assumption,  the primitive ideal space $\operatorname{Prim}(A)$ is finite (exactly happens when $A$ has finitely many ideals) and thus has a basis of compact-open sets. Therefore $A$ also has the ideal property (IP), whence $A$ is strongly purely infinite by  \cite[Proposition 2.11, 2.14]{P-R}
\end{proof}




In the rest of this section, we mainly show that existence of contractible sets for an action, which is common in certain boundary actions, usually implies pure infiniteness of the action.

\begin{defn}
	Let $G\curvearrowright X$ be a continuous action on a compact Hausdorff space $X$. An open set $V$ in $X$ is called \textit{contractible} if there exists an $x\in X$ such that for any neighborhood $U$ of $x$, there is a $g\in G$ such that $gV\subset U$.
\end{defn}

Let $G\curvearrowright X$ be a continuous action on a compact space. If the action is minimal and $X$ is not finite, then it has to be \textit{perfect} in the sense that there is no isolated points. Suppose the contrary, there exists an $x\in X$ such that $\{x\}$ is open. Then the minimality of the action implies that $X=G\cdot x$, which has to be finite. This is a contradiction.

\begin{prop}\label{prop: contractibility implying p.i.}
Let $X$ be an infinite compact Hausdorff space. If the action $\alpha:G\curvearrowright X$ is minimal and there is a contractible open set $V$, then $\alpha$ is purely infinite.
\end{prop}
\begin{proof}
	By Theorem \ref{thm: pure infiniteness of action}, it suffices to show $F\prec O$ for any compact set $F$ and non-empty open set $O$ in $X$. Indeed, let $F, O$ be such sets.
	First, since $\alpha$ is minimal, there are $g_1,\dots, g_n\in G$ such that $F\subset \bigcup_{i=1}^ng_iV$. In addition, because $X$ is perfect, one can choose $n$ disjoint non-empty open  subsets $O_1, \dots, O_n$ of $O$.

	On the other hand, let $x_0\in X$ be an element that makes $V$ contractible. Since $\alpha$ is minimal, for each $O_i$ where $1\leq i\leq n$, there is an $h_i\in G$ and a neighborhood $U_i$ of $x_0$ such that $h_iU_i\subset O_i$. Then the contractibility of $V$ implies that there is a $t_i\in G$ such that $t_iV\subset U_i$. This implies that $t_ih_iV\subset O_i$. Therefore, one has $\bigsqcup_{i=1}^nt_ih_ig^{-1}_i(g_iV)\subset O$, which establishes $F\prec O$.
\end{proof}

We also need the following permanence result later.

\begin{prop}\label{prop: p.i. in product with trivial action}
	Let $\alpha: G\curvearrowright X$ be a purely infinite action on an infinite locally compact Hausdorff space $X$ and $\beta: H\curvearrowright Y$ an action of a discrete group $H$ on a finite set $Y$, equipped with the discrete topology. Then the product action $\alpha\times\beta$ is still purely infinite.
	 If there is an $h\in H$ such that $h\neq e_H$ but $hy=y$ for some $y\in Y$. Then the product action $\alpha\times\beta$ is not topologically free. In particular, it is not essentially free.
\end{prop}
\begin{proof}
	Write $Y=\{y_1,\dots, y_m\}$ and it suffices to verify paradoxical comparison by Theorem \ref{thm: pure infiniteness of action}. For any non-empty open set $O\subset X\times Y$, there are open sets $O_i\subset X$ for each $1\leq i\leq m$ such that $O=\bigsqcup_{i=1}^mO_i\times\{y_i\}$ (some $O_i$ may be empty). One can do this by observing that $O_i\times \{y_i\}=O\cap \pi^{-1}_Y(\{y_i\})$. For any compact set $F\subset O$, using the same trick and the fact $\pi^{-1}_Y(\{y_i\})$ is clopen in $X\times Y$, one can find compact sets $F_i\subset O_i$ for $i=1,\dots, m$ such that $F=\bigsqcup_{i=1}^mF_i\times \{y_i\}$, where $\pi_Y$ is the natural projection from $X\times Y$ to $Y$. Then since $\alpha$ is purely infinite, for each $i$ with $O_i\neq \emptyset$, there are disjoint non-empty open sets $U_{i, 1}, U_{i, 2}\subset O_i$ such that  $F_i\prec U_{i, j}$ for both $j=1, 2$, which implies that there are collections of open sets $\{V^{i, j}_k\subset X: k=1,\dots, n_{i, j}\}$ and group elements $\{g^{i, j}_k\in G: k=1,\dots, n_{i, j}\}$ such that $F_i\subset \bigcup_{k=1}^{n_{i,j}}V^{i, j}_k$ and $\bigsqcup_{k=1}^{n_{i,j}}g^{i, j}_kV^{i, j}_k\subset U_{i, j}$ for $j=1, 2$. 
	
	We denote by $I=\{i\leq m: O_i\neq \emptyset\}$. Observe that $U_j=\bigsqcup_{i\in I}U_{i, j}\times \{y_i\}\subset O$ for $j=1,2$ and $U_1, U_2$ are disjoint non-empty open sets. In addition, one has 
	\[F=\bigsqcup_{i=1}^mF_i\times \{y_i\}\subset \bigcup_{i\in I}\bigcup_{k=1}^{n_{i,j}}(V^{i, j}_k\times\{y_i\})\]
	and
	\[\bigsqcup_{i\in I}\bigsqcup_{k=1}^{n_{i, j}}(g^{i, j}_k, e_H)(V^{i, j}_k\times \{y_i\})\subset \bigsqcup_{i\in I}U_{i, j}\times \{y_i\}=U_j\]
	for $j=1,2$. These establish that $\alpha\times \beta$ is purely infinite.
	
	Now, choose an $h\in H$ with $h\neq e_H$ and $y\in Y$ such that $hy=y$. Then $(e_G, h)$ fixes $X\times \{y\}$, which is open in $X\times Y$. This
	entails that $\alpha\times \beta$ is not topologically free and thus not essentially free.
\end{proof}

\section{Roller Boundary $\CR(X)$ and Nevo-Sageev Boundary $B(X)$}
In this section, we mainly study group actions on $\cat$ cube complexes $X$ and their Roller boundary $\CR(X)$ as well as a particular subset $B(X)$ of $\CR(X)$, which was introduced by Nevo and Sageev in \cite{N-S}. These two boundaries are of combinatorial nature. We begin with recalling the necessary concepts. We refer to \cite{N-S} and \cite{Le} for more information.

We denote by $\hat{\CH}$ the collection of all hyperplanes of $X$ and $\CH$ the set of all halfspaces. Similar to Stone-\u{C}ech compactification, one can use \textit{ultrafilers} consisting of certain halfspaces to define the \textit{Roller compactification}. See \cite{Roller}. Recall an ultrafilter $\alpha$ on $\CH$ is a subset of $\CH$ satisfying
\begin{enumerate}
	\item For any hyperplane $\hat{h}$, either $h\in \alpha$ or $h^*\in \alpha$ and
	\item If $h\in \alpha$ and $h\subset h'$ then $h'\in \alpha$.
\end{enumerate}
We denote by $\CU(X)$ the collection of all ultrafilters on $X$, which can be viewed as a closed subset of $\prod_{\hat{h}\in \hat{\CH}}\{h, h^*\}$ and thus a compact metrizable space if $X$ is locally finite. In addition, by identifying each vertex $x\in X^0$ by the  principal ultrafilter $\alpha_x=\{h\in \CH: x\in h\}$, the \textit{Roller boundary} is defined to be $\CR(X)=\CU(X)\setminus X^0$, which is also a compact metrizable space if $X$ is locally finite. Nevo and Sageev in \cite{N-S} also consider the following special subset $B(X)$ of $\CR(X)$, which is referred as the \textit{Nevo-Sageev boundary} in this paper. Consider the following set $\CU_{NT}(X)$ consisting all non-terminating ultrafilters:
\[\CU_{NT}(X)=\{\alpha\in\CU(X): h\in \alpha \Rightarrow \text{there exsits } h'\in \alpha\ \text{with }h'\subsetneq h\}\]
and define $B(X)=\overline{\CU_{NT}(X)}$ in $\CU(X)$. Such a boundary is always non-empty if $X$ is essential and cocompact by \cite[Theorem 3.1]{N-S}. Unlike the visual boundary that will be addressed in the next section, $B(X)$ has a very nice property that if $X$ is not irreducible and decomposes as $X=\prod_{i=1}^nX_i$, then $B(X)=\prod_{i=1}^nB(X_i)$ so that the dynamics on the $B(X)$ is more convenient to deal with.

Let $\Gamma$ be a finite simple graph and $W_\Gamma$ and $A_\Gamma$ be its RACG and RAAG, respectively.  Following the notation in \cite{Char}, we denote by $\Sigma_\Gamma$  the Davis complex for $W_\Gamma$ and $\tilde{S_\Gamma}$ the universal cover of the Salvetti complex $S_\Gamma$ for $A_\Gamma$. Note that both of $\Sigma_\Gamma$ and $\tilde{S_\Gamma}$ are finite dimensional $\cat$ cube complexes and the $1$-skeleton of them are exactly the Cayley graph of $W_\Gamma$ and $A_\Gamma$, respectively. From this view, for Davis complex $\Sigma_\Gamma$, the set of vertices $C\subset W_\Gamma$ spans a cube if and only if it forms a coset $wW_\Lambda$ for some \textit{finite} special subgroup $W_\Lambda$. Here $\Lambda$ is allowed to be empty in which case $wW_\Lambda=\{w\}$.

Such a graph theoretical description of $\tilde{S_\Gamma}$ is analogous to that of $\Sigma_\Gamma$. In fact, for any special subgroup $W_\Lambda$ of the Coxeter group $W_\Gamma$, with generating vertices $\{v_1,\dots, v_k\}\subset V$, one can lift $W_\Lambda$ to the subset $\widehat{W_\Lambda}\subset A_\Lambda$ (not subgroup) consisting of elements of the form $v^{\epsilon_1}_1\cdots v^{\epsilon_k}_k$ where $\epsilon_i=0, 1$. Then $\tilde{S_\Gamma}$ is the cube complex whose $1$-skeleton is the Cayley graph of $A_\Gamma$ and whose cubes are set of vertices of the form $a\widehat{W_\Lambda}$ for some $a\in A_\Gamma$ and some finite special subgroup $W_\Lambda$. 

Both $\Sigma_\Gamma$ and $\tilde{S_\Gamma}$ have very nice properties. write $X_\Gamma=\Sigma_\Gamma$ or $\tilde{S_\Gamma}$ for simplicity.  First,  $X_\Gamma$ is always cocompact since $G_\Gamma$ acting on it cocompactly. Using the fact that $1$-skeleton of $X_\Gamma$ is exactly the Cayley graph of the corresponding group $W_\Gamma$ or $A_\Gamma$, the edges $X^1_\Gamma$ of $X_\Gamma$ naturally are labeled by the vertex set $V$ of the defining graph $\Gamma$. To observe more, for any vertex $x\in X_\Gamma^0$ ($x$ thus actually belongs to $W_\Gamma$ or $A_\Gamma$), the $1$-skeleton of the link of $x$ is isomorphic to the defining graph $\Gamma$. In addition, the edges in the $1$-skeleton of the link of $x$ labeled by a subset $\{v_{k_1}, \dots, v_{k_m}\}$ of $V$ belong to the same cube if and only if $(v_{k_i}, v_{k_j})\in E$ for any $1\leq i, j\leq m$. Moreover, the labels of edges dual to one hyperplane $\hat{h}$ are all the same $v\in V$, which is called the \textit{type} of $\hat{h}$. Finally, If in $X_\Gamma$, two hyperplanes $\hat{h}_1\cap \hat{h}_2\neq\emptyset$, then their types $v_1$ and $v_2$ satisfy $(v_1, v_2)\in E$. See more in \cite[Section 8]{Le}. For essentialness of $X_\Gamma$, the following theorem was also proved in \cite{Le}. For a simple graph $\Gamma=(V, E)$, the graph $\Gamma^c=(V, E^c)$, in which $E^c=\{(u, v):u, v\in V\ \textit{and }  (u, v)\notin E\}$ is called its  \textit{complemented graph}.

\begin{prop}\cite[Proposition 8.1, Lemma 8.3]{Le}\label{prop: essential racg and raag}
Let $\Gamma$ be a finite simple graph. The Davis simplex $\Sigma_\Gamma$ of the corresponding RACG $W_\Gamma$ is essential if and only if the complement graph $\Gamma^c$ does not have an isolated vertex. In addition, the universal cover of the Salvetti complex $\tilde{S_\Gamma}$ of the corresponding RAAG $A_\Gamma$ is always essential.
\end{prop}

Therefore $B(\tilde{S_\Gamma})$ is always not empty  and so are  $B(\Sigma_\Gamma)$ whenever $\Gamma^c$ has no isolated vertex. We now study the reducibility of  $\tilde{S_\Gamma}$ and $\Sigma_\Gamma$. Let $\Gamma=(V, E)$. Simply observe that $\Gamma^c$ has an isolated vertex $v$ if and only if $(v, w)\in E$ for any $w\in V\setminus\{v\}$ if and only if $W_\Gamma$ has a factor of $\Z_2$ as a special subgroup, i.e., $W_\Gamma=W_{\Gamma'}\times \Z_2$ for some subgraph $\Gamma'$. Therefore, the Davis complex $\Sigma_\Gamma$ is essential if and only if $W_\Gamma$ has no factor $\Z_2$ as a special subgroup. 

\begin{lem}\cite[Lemma 2.5]{C-S}\label{lem: direct product}
	A decomposition of a $\cat$ cube complex $X$ as a product of cube complexes corresponds to a partition of the collection of hyperplanes of $X$, $\hat{H}=\hat{H}_1\sqcup \hat{H}_2$ such that every hyperplane in $\hat{H}_1$ meets every hyperplane in $\hat{H}_2$.
\end{lem}

Now we have the following result, which seems well-known to experts. However, to be self-contained, we include the proof here.

\begin{prop}\label{prop: decomposition cub cpx}
	Let $\Gamma=(V, E)$ be a simple finite graph with $|V|\geq 2$ and $X_\Gamma=\Sigma_\Gamma$ or $\tilde{S_\Gamma}$. Then $X_\Gamma$ can be written as a direct product of $\cat$ cube complexes if and only if $\Gamma$ is a join.
\end{prop}
\begin{proof}
	Suppose $X$ is a direct product. Then Lemma \ref{lem: direct product} implies that one can partition the whole $\hat{H}=\hat{H}_1\sqcup \hat{H}_2$ non-trivially such that $\hat{h}_1\cap\hat{h}_2\neq \emptyset$ for any  $\hat{h}_1\in \hat{H}_1$ and $\hat{h}_2\in \hat{H}_2$.  Now, observe that there are $\hat{h}_{1}\in \hat{H}_1$  and $\hat{h}_{2}\in \hat{H}_2$ with different types $v$ and $w$, respectively.  Otherwise, there is only one type for all hyperplanes in $\hat{H}$, which means $|V|=1$. This is a contradiction.  
	
	We then define $J_{1, 0}=\{\hat{h}_1\}$ and $J_{2, 0}=\{\hat{h}_2\}$. Now we enumerate $V\setminus \{v, w\}$ by $\{v_1, v_2,\dots, v_n\}$. Suppose we have defined $J_{1, m}$ and $J_{2, m}$ for $0\leq m<n$. Then for $v_{m+1}$, choose $\hat{h}\in \hat{H}$ with type $v_{m+1}$. Observe that either  $\hat{h}\in \hat{H}_1$ or $\hat{H}_2$. Then for $i, j=1, 2$ and $i\neq j$, define $J_{i, m+1}=J_{i ,m}\cup\{\hat{h}\}$ if $\hat{h}\in \hat{H}_i$ and $J_{j, m+1}=J_{j, m}$. 
	
	Finally one defines $V_i=\{v\in V: \text{there is a }\hat{h}\in J_{i, n} \text{ of type }v\}$ for $i=1, 2$. Our construction implies $V_1\sqcup V_2=V$. In addition, because $J_{i, n}\subset \hat{H}_i$ for $i=1, 2$, one has $\hat{h}_1\cap \hat{h}_2\neq \emptyset$ for any $\hat{h}_i\in J_{i, n}$ for $i=1,2$, which implies that $(v_1, v_2)\in E$ for any $v_1\in V_1$ and $v_2\in V_2$. This means
	that $\Gamma$ is a join. The converse direction is trivial by observing $X_\Gamma=X_{\Gamma_1}\times X_{\Gamma_2}$ whenever $\Gamma$ has a non-trivial join $\Gamma=\Gamma_1\star\Gamma_2$ for some subgraphs $\Gamma_1$ and $\Gamma_2$.
\end{proof}

Now if $X_\Gamma$ can be written as a non-trivial direct product, then $\Gamma$ has a non-trivial join, e.g., $\Gamma=\Gamma_1\star \Gamma_2$, which implies $X_\Gamma=X_{\Gamma_1}\times X_{\Gamma_2}$. If one of these $\Gamma_i$, $i=1, 2$, is still a join, we can decompose $X_{\Gamma_i}$ further in the same manner. Following this strategy, since $\Gamma$ is finite, one can decompose $X_\Gamma=X_{\Gamma_1}\times\dots\times X_{\Gamma_m}$, in which each factor is irreducible. We remark that such a factorization is unique up to a permutation of factors. See the proof of \cite[Proposition 2.6]{C-S}. Then the natural action of $W_\Gamma\curvearrowright \Sigma_\Gamma$ is exactly the product of all actions $W_{\Gamma_i}\curvearrowright \Sigma_{\Gamma_i}$, i.e. $W_\Gamma=W_{\Gamma_1}\times\dots \times W_{\Gamma_m}\curvearrowright \Sigma_{\Gamma_1}\times\dots \times \Sigma_{\Gamma_m}$ coordinatewise.  The same also holds for $A_\Gamma$. We now establish the following result on the structure of $X_\Gamma$. We first recall concepts in \cite[Section 4.4]{C-S} that a $\cat$ cube complex $X$ is called $\R$-\textit{like} if there is an $\aut(X)$-invariant geodesic line $\ell\subset X$. In addition, we remark that if $\aut(X)$ acts cocompactly, then $X$ is quasi-isometric to the real line $\R$.

\begin{lem}\label{lem: structure of cpx}
Let $X_\Gamma=\Sigma_\Gamma$ or $ \tilde{S_\Gamma}$ such that $X_\Gamma$ is essential. Let $X_\Gamma=X_{\Gamma_1}\times\dots\times X_{\Gamma_m}$ be a decomposition of $X_\Gamma$ into irreducible factors described above. Suppose one $X_{\Gamma_i}$ is Euclidean. Then 
\begin{enumerate}
\item in the case $X_{\Gamma_i}=\Sigma_{\Gamma_i}$ one has $W_{\Gamma_i}\simeq D_\infty$; and
\item in the case  $X_{\Gamma_i}=\tilde{S_{\Gamma_i}}$ one has $A_{\Gamma_i}\simeq \Z$.
\end{enumerate}
\end{lem}
\begin{proof}
Let $X_\Gamma=\Sigma_\Gamma$ or $ \tilde{S_\Gamma}$ with a decomposition $X_\Gamma=X_{\Gamma_1}\times\dots\times X_{\Gamma_m}$, where each $X_{\Gamma_i}=\Sigma_{\Gamma_i}$ or $ \tilde{S_{\Gamma_i}}$ is irreducible. Since $X_\Gamma$ is essential, Proposition \ref{prop: essential racg and raag} implies that each factor $X_{\Gamma_i}$ is essential and thus unbounded.

Write $G_i=W_{\Gamma_i}$ or $A_{\Gamma_i}$, respectively for simplicity. Suppose a factor $X_{\Gamma_i}$ is Euclidean. Then there is a $\aut({X_{\Gamma_i}})$-invariant flat in $X_{\Gamma_i}$. Because $X_{\Gamma_i}$ is essential and the $1$-skeleton of $X_{\Gamma_i}$ is exactly the Cayley graph of $G_i$, on which the action of $G_i$ is transitive,  the $G_i$-action on $X_{\Gamma_i}$ is essential and cocompact. Therefore, $\aut(X_{\Gamma_i})$ acts on $X_{\Gamma_i}$ essentially and cocompactly since $G_i\leq \aut(X_{\Gamma_i})$. Then \cite[Lemma 7.1]{C-S} implies that $X_{\Gamma_i}$ is $\R$-like and thus $X_{\Gamma_i}$ is quasi-isometric to the real line $\R$. This implies that $G_i$ is quasi-isometric to $\Z$ by Lemma
\ref{lem: quasi-isometric} and thus $G_i$ is virtually $\Z$ and has exactly two ends.

In the case $G_i=W_{\Gamma_i}$, applying \cite[Theorem 8.7.3]{Dav}, one has that $G_i$ is the product of a finite special subgroup and a speicial subgroup that is an infinite dihedral group $D_\infty$.  Now if the finite special group factor of $G_i$ is nontrivial, then $X_{\Gamma_i}$ is reducible by Proposition \ref{prop: decomposition cub cpx}, which is a contradiction to the fact that $X_{\Gamma_i}$ is irreducible. This implies $G_i=D_\infty$. In the case $G_i=A_{\Gamma_i}$, for any vertices $v, w$ of the graph $\Gamma_i$, the two-generator subgroup $\langle v, w\rangle \leq A_{\Gamma_i}$ has to be either free or abelian by a classical result of Baudisch in \cite{Bau}, which implies that $v, w$ have to commute because $A_{\Gamma_i}$ is virtually $\Z$, which is amenable. This implies that $ A_{\Gamma_i}$ has to be isomorphic to $\Z$.
\end{proof}

These lead to the following result.

\begin{thm}\label{thm: decompostion}
	Let $G_\Gamma\curvearrowright X_\Gamma$ where $G_\Gamma=W_\Gamma$ or $A_\Gamma$ and $X_\Gamma=\Sigma_\Gamma$ or $\tilde{S_\Gamma}$, respectively. Then $G_\Gamma=G_{\Gamma'}\times H^n$ for some subgraph $\Gamma'$ of $\Gamma$ and a group $H$ and an $n\in \N$, in which  $H=D_\infty$ if $G_\Gamma=W_\Gamma$ and $H=\Z$ if $G_\Gamma=A_\Gamma$. In addition, the corresponding complex $X_{\Gamma'}$ of $G_{\Gamma'}$ is strictly non-Euclidean.  
\end{thm}
\begin{proof}
For the action $G_\Gamma\curvearrowright X_\Gamma$, by the reduction, it can be written as \[G_{\Gamma_1}\times\dots \times G_{\Gamma_m}\curvearrowright X_{\Gamma_1}\times\dots \times X_{\Gamma_m},\]
in which each $X_{\Gamma_i}$ is irreducible. Collecting all Euclidean factors $X_{\Gamma_i}$ together.  Without loss of generality, one may assume they are exactly the final $n$ factors. Then Lemma \ref{lem: structure of cpx} implies that for $m-n+1\leq i\leq m$, the group $H=G_{\Gamma_i}\simeq D_\infty$ or $\Z$ depending on which case  is under consideration.  Now $\Gamma'$ is defined to be the join of all  graphs $\Gamma_1, \dots, \Gamma_{m-n}$ such that $X_{\Gamma'}=X_{\Gamma_1}\times\dots \times X_{\Gamma_{m-n}}$, which is strictly non-Euclidean by definition.
\end{proof}

On the other hand, ultrafilters and the Nevo-Sageev boundary work compatible with the product of $\cat$ cube complexes. See \cite{N-S}. Given a decomposition $X\simeq\prod_{i=1}^mX_i$, one actually has $\CU(X)\simeq \prod_{i=1}^m\CU(X_i)$ and $B(X)\simeq \prod_{i=1}^mB(X_i)$.   In our case,  since the action $G_\Gamma\curvearrowright X_\Gamma$ action can be decomposed to be \[G_{\Gamma_1}\times\dots \times G_{\Gamma_m}\curvearrowright X_{\Gamma_1}\times\dots \times X_{\Gamma_m}\]
Then the action $G_\Gamma\curvearrowright \CU(X_\Gamma)$ is exactly the product action \[G_\Gamma=\prod_{i=1}^mG_{\Gamma_i}\curvearrowright \prod_{i=1}^m\CU(X_{\Gamma_i})\] and  therefore  the action on the Nevo-Sageev boundary $G_\Gamma\curvearrowright B(X_\Gamma)$ can be written as  \[G_\Gamma=\prod_{i=1}^mG_{\Gamma_i}\curvearrowright \prod_{i=1}^mB(X_{\Gamma_i}).\]  
Now we study the dynamics of $G_\Gamma$ on the Nevo-Sageev boundary $B(X_\Gamma)$. In the irreducible Euclidean case, let $G_\Gamma=D_\infty$ or $\Z$. Then by a simple observation, the corresponding Roller boundary and the Nevo-Sageev boundary $\CR(X_\Gamma)=B(X_\Gamma)$ is a set consisting exactly two points, which can be identified by the only two infinite geodesics in this case. However, the action on them are different. In the RACG case,  $W_\Gamma=D_\infty$ when the defining graph $\Gamma=(V=\{u, v\}, E=\emptyset)$.  If we denote by $\breve{0}$ the infinite geodesic  $uvuv\dots$ and $\breve{1}$ the  geodesic $vuvu\dots$ for simplicity, then $\CR(\Sigma_\Gamma)=B(\Sigma_\Gamma)$ can be identified by $\{\breve{0}, \breve{1}\}$ and therefore the action of $W_\Gamma$ on the boundary is generated by the permutations $u\cdot \breve{0}=\breve{1}$ and $u\cdot \breve{1}=\breve{0}$ as well as $v\cdot \breve{1}=\breve{0}$ and $v\cdot \breve{0}=\breve{1}$. In the RAAG case that $A_\Gamma=\Z$, it is easy to see the the boundary $\CR(\tilde{S_\Gamma})=B(\tilde{S_\Gamma})$ are exactly the two ends of it, on which the action of $A_\Gamma=\Z$ is trivial.  Now, we write the boundary $B(X_\Gamma)=\{\breve{0}, \breve{1}\}$ for simplicity in both cases. 

In the strictly non-Euclidean case, the following  theorem was established in \cite{N-S}. See \cite[Theorem 5.1]{N-S} and the proof of \cite[Theorem 5.8]{N-S}.

\begin{thm}\cite[Theorem 5.1, Theorem 5.8]{N-S}\label{thm: contratible set and boundary}
Let $X$ be an essential strictly non-Euclidean $\cat$ cube complex admitting a proper cocompact action of $G\leq \aut(X)$. Then $B(X)$ is a $G$-boundary and there is a contractible open set for the action in $B(X)$.
\end{thm}

The following theorem was proved in \cite{B-C-G-N-W}. See also \cite[Theorem 7.4]{N-S}

\begin{thm}\cite[Theorem 4.2]{B-C-G-N-W}\label{thm: amenable stabilizer on B(X)}
	Let countable discrete group $G$ acts properly on a finite-dimensional $\cat$ cube complex $X$. Then the stabilizer group $\stab_G(x)$ is amenable for any  $x\in \CR(X)$. In particular, $\stab_G(x)$ is amenable for any  $x\in B(X)$.
\end{thm}

Recall a classical result that any infinite irreducible  RACG $W_\Gamma$ ($\Gamma$ is finite) is $C^*$-simple whenever $|V(\Gamma)|\geq 3$ (see e.g., \cite{Harpe}). We then establish the following result for irreducible RAAGs, which might be known to experts.

\begin{lem}\label{lem: C simple raag}
Let $A_\Gamma\neq \Z$ be an irreducible  RAAG in which the defining graph $\Gamma=(V, E)$ is finite. Then $A_\Gamma$ is $C^*$-simple.
\end{lem}
\begin{proof}
We follow the embedding arguments in \cite{D-J} to show that such a $A_\Gamma$ can be embedded into an irreducible non-amenable RACG. Indeed, since $A_\Gamma$ is irreducible, the defining graph $\Gamma$ has no joins. Define a new graph $\Gamma'=(V', E')$ in the way that the vertex set $V'=V\times \{0, 1\}$ and 
\begin{enumerate}
	\item $((v, 1), (w, 1))\in E'$ if and only if $(v, w)\in E$;
	\item $((v, 0), (w, 0))\in E'$ for any $v, w\in V$; and 
	\item $((v, 0), (w, 1))\in E'$ if and only if $v\neq w$.
\end{enumerate}
It was proved in \cite{D-J} that $A_\Gamma$ can be embedded in $W_{\Gamma'}$ as a subgroup with a finite index. We claim that $W_{\Gamma'}$ is irreducible by showing that $\Gamma'$ has no joins. Suppose the contrary that $\Gamma'=\Gamma'_1\star\Gamma'_2$ for two non-trivial subgraph $\Gamma'_1=(V'_1, E'_1)$ and $\Gamma'_2=(V'_2, E'_2)$. Then for any pair of vertices $\{v\}\times \{0, 1\}$, one has either $\{v\}\times \{0, 1\}\subset V'_1$ or $\{v\}\times \{0, 1\}\subset V'_2$ because there is no edge in $E'$ between $(v, 0)$ and $(v, 1)$ in $\Gamma'$. This implies that $V_i=\{v\in V: \{v\}\times \{0, 1\}\subset V'_i\}$ for $i=1,2$ form a non-trivial partition of $V$. Now for any $v\in V_1$ and $w\in V_2$, since $\Gamma'_1$ and $\Gamma'_2$ form a join, one has $((v, 1), (w, 1))\in E'$, which implies that $(v, w)\in E$. Therefore $\Gamma$ itself has a join, which is a contradiction.
Now since $A_\Gamma\neq \Z$, then $|V|\geq 2$ and thus $|V'|\geq 4$ by our construction. Now because $W_{\Gamma'}$ is also irreducible, one has $W_{\Gamma'}$ is $C^*$-simple, whence $A_\Gamma$ is also $C^*$-simple because $A_\Gamma$ is a subgroup of $W_{\Gamma'}$ with a finite index by \cite[Proposition 19]{Harpe}.
\end{proof}

\begin{rmk}\label{rmk: C simple and non-Euclidean}
Let $\Gamma=(V, E)$ be a finite simple graph. In the infinite irreducible RACG case, if $|V|=2$ then $W_\Gamma=D_\infty$ and then $\Sigma_\Gamma$ is Euclidean. In the irreducible RAAG case, if $|V|=1$ then $A_\Gamma=\Z$ and the complex $\tilde{S_\Gamma}$ is Euclidean. Now, we still write $G_\Gamma\curvearrowright X_\Gamma$  for simplicity as usual, where $G_\Gamma=W_\Gamma$ or $A_\Gamma$ and $X_\Gamma=\Sigma_\Gamma$ or $\tilde{S_\Gamma}$, respectively. Therefore, If $X_\Gamma$ is irreducible and non-Euclidean then $G_\Gamma$ is $C^*$-simple.
\end{rmk}

Now we are ready to prove the following main theorem in this section. We remark that in the following theorem it is necessary to assume there is one non-Euclidean factor $X_{\Gamma_i}$ in the canonical decomposition $X_\Gamma=X_{\Gamma_1}\times\dots\times X_{\Gamma_m}$ discussed above because if not, as we have shown above, the boundary $B(X_\Gamma)$ will be finite and the group $G_\Gamma$ is amenable, in which case the action $\beta$ below cannot be purely infinite.

\begin{thm}\label{thm: pure inf on B(X)}
	Let $G_\Gamma\curvearrowright X_\Gamma$ where $X_\Gamma$ is essential and has at least one non-Euclidean irreducible factor $X_{\Gamma_i}$ in the canonical decomposition above. Then the induced action $\beta: G_\Gamma\curvearrowright B(X_\Gamma)$ is purely infinite.  In addition, in the RAAG case, the action $\beta$ has finitely many $G_\Gamma$-invariant closed sets. In the RACG case, $\beta$ is minimal. However, if $G_\Gamma$ has special subgroups $D_\infty$ or $\Z$, depending on the RACG case or the RAAG case, then the action is not topologically free.
\end{thm}
\begin{proof}
	Theorem \ref{thm: decompostion} implies that $G_\Gamma=G_{\Gamma'}\times H^n$ where $H=D_\infty$ or $\Z$ and the complex $X_{\Gamma'}$ is strictly non-Euclidean. Observe that $\beta$ is exactly the product action of $\beta_1: G_{\Gamma'}\curvearrowright B(X_{\Gamma'})$ and the action $\beta_2: H^n\curvearrowright \{\breve{0}, \breve{1}\}^n$. Theorem \ref{thm: contratible set and boundary} implies that there is a contractible set $V$ in $B(X_{\Gamma'})$ for $\beta_1$ and also $\beta_1$ is minimal. Then Proposition \ref{prop: contractibility implying p.i.} entails that $\beta_1$ is purely infinite. In addition, because $\beta_2$ is an action on a finite set, Proposition \ref{prop: p.i. in product with trivial action} implies that $\beta=\beta_1\times \beta_2$ is purely infinite. 
	
	Now observe that $G_{\Gamma'}$ is a finite direct product of $C^*$-simple groups by Remark \ref{rmk: C simple and non-Euclidean} and thus $G_{\Gamma'}$ itself is $C^*$-simple by \cite[Proposition 19]{Harpe}. Now since $\beta_1$ is a $G_{\Gamma'}$-boundary action by Theorem \ref{thm: contratible set and boundary} and the stabilizer group $\stab_G(x)$ of each $x\in B(X)$ is amenable by Theorem \ref{thm: amenable stabilizer on B(X)}, one has $\beta_1$ is topologically free by \cite[Proposition 1.9]{B-K-K-O}.
	
	Now if $n>0$, In the RAAG case, note that $H=\Z$ and the action $\beta_2$ is the trivial action.   Thus $\beta$ has in total $2^n$ $G_{\Gamma}$-invariant closed sets.  In the RACG case, note that the corresponding action $\beta_2$ of $H=D_\infty$ on $\{\breve{0}, \breve{1}\}^n$ is minimal. Then it is easy to verify that $\beta$ is minimal.
	
	However, because there is an non-trivial element in $H$ fixing a point in $\{\breve{0}, \breve{1}\}^n$, it follows from Proposition \ref{prop: p.i. in product with trivial action} that $\beta$ is not topologically free. In particular, $\beta$ is not essentially free.
\end{proof}

Applying Theorem \ref{thm: pure inf main}, one immediately has the following theorem.

\begin{thm}\label{thm: main two}
Let $G_\Gamma\curvearrowright X_\Gamma$ where $X_\Gamma$ is essential and has at least one non-Euclidean irreducible factor $X_{\Gamma_i}$ in the decomposition above.  Suppose 
\begin{enumerate}
	\item  $G_\Gamma=W_\Gamma$ has no special subgroup $D_\infty$ in the  RACG case;
	\item  $G_\Gamma=A_\Gamma$ has no special subgroup $\Z$  in the RAAG case.
\end{enumerate}
Then the reduced crossed product $A=C(B(X_\Gamma))\rtimes_r G_\Gamma$ of the induced action on the Nevo-Sageev boundary $\beta: G_\Gamma\curvearrowright B(X_\Gamma)$ is unital simple separable and purely infinite. In addition, in the RACG case, $A$ is nuclear as well and thus a Kirchberg algebra satisfying the UCT.
\end{thm}
\begin{proof}
Theorem \ref{thm: decompostion} implies that $G_\Gamma=G_{\Gamma'}\times H^n$ where $H=D_\infty$ or $\Z$ and the complex $X_{\Gamma'}$ is strictly non-Euclidean. Now if $G_\Gamma$ has no special subgroups of $D_\infty$ or $\Z$, then $n=0$ and thus $X_\Gamma$ is strictly non-Euclidean. Then Theorem \ref{thm: pure inf on B(X)} implies that $\beta$ is minimal topologically free and purely infinite and thus $A=C(B(X_\Gamma))\rtimes_r G_\Gamma$  is unital simple separable and purely infinite by Theorem \ref{thm: pure inf main}. 

It is left to show $A$ is nuclear in the RACG case. First, for the Davis complex $X_\Gamma=\Sigma_\Gamma$, in the irreducible case,  it was proved by \cite[Theorem D]{Le} that $\CR(\Sigma_\Gamma)=B(\Sigma_\Gamma)$. In addition, $\CR(\Sigma_\Gamma)$ can be identified with the horofunction boundary of the Cayley grpah of $W_\Gamma$ with the usual $\ell_1$-metric by a unpublished work of U. Bader and D. Guralnik (see e.g. \cite[Section 1.3]{N-S}). On the other hand, the horofunction boundary of the Cayley grpah of $W_\Gamma$ with the usual $\ell_1$-metric also coincides with the minimal combinatorial boundary $\mathscr{C}_1(\Sigma_\Gamma)\setminus \Sigma_\Gamma$ introduced in \cite{C-L}. See \cite[Theorem 3.1]{C-L}. Then finally it was proved in \cite{Lec}  that the action of $W_\Gamma$ on the minimal combinatorial boundary $\mathscr{C}_1(\Sigma_\Gamma)\setminus \Sigma_\Gamma$ is amenable. Therefore, in the irreducible case, the action $\beta: W_\Gamma\curvearrowright B(\Sigma_\Gamma)$ is amenable using this chain of identifications. However, in our more general case that $X_\Gamma$ is strictly non-Euclidean, the action $\beta$ is a product of amenable actions and thus amenable. Therefore, $A$ is nuclear and thus a Kirchberg algebra. Also, $A$ satisfies the UCT by the result of Tu in \cite{Tu} . 
\end{proof}

If $G_\Gamma$ have sepcial subgroups of $D_\infty$ or $\Z$, we then have the following result on the structure of the reduced crossed products.

\begin{cor}\label{cor: full result on B(X)}
	Let $G_\Gamma=G_{\Gamma'}\times H^n$ be the decomposition in Theorem \ref{thm: decompostion}. Write $A=C(\partial X_\Gamma)\rtimes_r G_\Gamma$. Then one has
	\begin{enumerate}
		\item In the RACG case,  one has $A=(C(\partial X_{\Gamma'})\rtimes_r G_{\Gamma'})\otimes \bigotimes_{i=1}^n (C(\{\breve{0}, \breve{1}\})\rtimes_r D_\infty)$, where $C(X_{\Gamma'})\rtimes G_{\Gamma'}$ is unital simple separable purely infinite.
		\item In the RAAG case, one has $A=(C(\partial X_{\Gamma'})\rtimes_r G_{\Gamma'})\otimes C(\{\breve{0}, \breve{1}\}^n)\otimes C(\T^n)$ in which $\T$ is the unit circle and $C(X_{\Gamma'})\rtimes_r G_{\Gamma'}$ is unital simple separable purely infinite.
	\end{enumerate}
In the RACG case, $A$ is $\CO_\infty$-stable and actually, in either case, $A$ is strongly purely infinite.
\end{cor}
\begin{proof}
In the context of the proof of Theorem \ref{thm: pure inf on B(X)}, one can decompose $\beta: G_\Gamma\curvearrowright B(X_\Gamma)$ into the product action of $\beta_1: G_{\Gamma'}\curvearrowright B(X_{\Gamma'})$ and  $\beta_2: H^n\curvearrowright \{\breve{0}, \breve{1}\}^n$, which implies that $C(\partial X_\Gamma)\rtimes_r G_\Gamma$ is the tensor product of the reduced product  $C(\partial X_{\Gamma'})\rtimes_r G_{\Gamma'}$ of $\beta_1$ with the reduced crossed product of $\beta_2$. Note that $C(\partial X_{\Gamma'})\rtimes_r G_{\Gamma'}$ is unital simple separable purely infinite by Theorem \ref{thm: main two}. 

On the other hand, in the RACG case, the reduced crossed products of $\beta_2$ is exactly $\bigotimes_{i=1}^n (C(\{\breve{0}, \breve{1}\})\rtimes_r D_\infty)$. In the RAAG case, because the action $\beta_2$ is trivial, the corresponding crossed product is $C(\{\breve{0}, \breve{1}\}^n)\otimes C^*_r(\Z^n)=C(\{\breve{0}, \breve{1}\}^n)\otimes C(\T^n)$.

In the RACG case, because $C(\partial X_{\Gamma'})\rtimes_r G_{\Gamma'}$ is a Kirchberg algebra and thus $\CO_\infty$-stable. Therefore, $A$ is also $\CO_\infty$-stable and thus strongly purely infinite. In the RAAG case, $C(\partial X_{\Gamma'})\rtimes_r G_{\Gamma'}$ is simple and purely infinite, and thus strongly purely infinite by \cite[Theorem 9.1]{Kir-Rord}.  Then it follows from \cite[Theorem 1.3]{Kir-S} that $A$ is strongly purely infinite as well.
\end{proof}

\begin{rmk}\label{rmk: aukward}
	We remark that Theorem \ref{thm: main two}(1) has generalized the result in \cite[Example 4.8]{G-G-K-N} because the boundary considered there can also be identified with the horofunction boundary of the Cayley grpah of $W_\Gamma$ with the usual $\ell_1$-metric.  which is exactly the Nevo-Saggev boundary as we have shown in the proof of Theorem \ref{thm: main two}. On the other hand, after the minimality and topologically freeness as well as the amenability of the action $\beta$ have been established by using the method in this section, one can apply the result \cite[Theorem B]{G-G-K-N} to obtain a different proof of this result.  However, at this moment, the results in \cite{G-G-K-N} do not apply to \ref{thm: main two}(2). The main obstruction is that,  to the best knowledge of the authors, it is unknown whether the action of an irreducible non-amenable RAAG on the Nevo-Sageev boundary $B(X)$ is amenable.
\end{rmk}

\section{Visual Boundary $\partial X$}
\subsection{Actions on irreducible $\cat$ cube complexes}

In this section, we focus on actions of certain groups on the visual boundaries of some $\cat$ spaces, especially on the visual boundaries of $\cat$ cube complexes in this subsection. We will deal with actions on the visual boundaries of trees in the next subsection.

We begin with the definition of the visual boundary. Let $(X, d)$ be a $\cat$ space, we say two geodesic rays $c_1, c_2: [0, \infty)\to X$ are \textit{asymptotic} if there is a $C>0$ such that $d(c_1(t), c_2(t))<C$ for any $t\in [0, \infty)$. We remark that being asymptotic is an equivalence relation for geodesic rays. Denote by $\partial X$ the set of equivalence classes, which is called the boundary set of $X$. In addition, for any geodesic $c: [0, \infty)\to X$, we denote by $c(\infty)$ the equivalence class containing $c$. We further remark that \cite[Proposition I. 8.2]{B-H} shows that for any geodesic $\alpha: [0, \infty)\to X$ and any $x\in X$, there is a unique geodesic ray $\beta$ starting at $x$ with $\beta(\infty)=\alpha(\infty)$. Therefore, it suffices to consider all geodesic rays starting at a chosen base point  $x_0\in X$. Let $\alpha$ be a geodesic ray starting at $x_0$ and $r, \epsilon>0$.  Consider the following set 
\[U_{\alpha(\infty), r, \epsilon}=\{\beta(\infty)\in \partial X: \beta(0)=x_0\ \text{and } d(\alpha(t), \beta(t))<\epsilon\ \text{for all }t<r\}.\]
Note that such sets form a neighborhood basis for the geodesic $\alpha$ and thus all such sets induce a compact metrizable topology on $\partial X$, which is called the \textit{cone topology}. Finally, it was proved in \cite[Proposition I. 8.8]{B-H}  that the cone topology is independent of the choice of the base point.

\begin{defn}
	Let $(X, d)$ be a $\cat$ metric space.  The boundary set $\partial X$ is called the \textit{visual boundary} of $X$ when equipped with the cone topology. 
\end{defn}

If $X$ is Gromov hyperbolic, then $\partial X$ is exactly the classical Gromov boundary of $X$. Denote by $\isom(X)$ the isometry group on $X$. It goes back to Gromov \cite[Section 8.2]{Gro} that any hyperbolic (loxodromic) element $g\in \isom(X)$ performs so-called \textit{north-south dynamics} on the boundary $\partial X$ as a homeomorphism in the following sense.

\begin{defn}
	Let $X$ be a topological space and $g$ is an homeomorphism of $X$. We say $g$ has north-south dynamics with respect to two fixed points $x, y\in X$, which are called \textit{attracting} and \textit{repelling} fixed points of $g$, respectively, if for any open neighborhoods $U$ of $x$ and $V$ of $y$, there is an $m\in \N$ and such that $g^m(X\setminus V)\subset U$ and $g^{-m}(X\setminus U)\subset V$.
\end{defn}

Now let $X$ be a proper $\cat$ space, which is not necessary hyperbolic. It is actually known that any rank-one isometry $g\in \isom(X)$ has the north-south dynamics on the visual boundary $\partial X$ (see e.g. \cite{B-B} and \cite{C-S}).  In general, North-south dynamics have a very strong flavor of pure infiniteness. Indeed, we have the following proposition. We remark that similar arguments also appeared in \cite{A-D} and \cite{L-S}. 
\begin{prop}\label{prop: north-south visual}
	Let $\alpha:G\curvearrowright Y$ be a continuous minimal action of a discrete group $G$ on a infinite compact Hausdorff space $Y$. Suppose there is a $g\in G$  performing the north-south dynamics. Then $\alpha$ is 2-filling and thus is purely infinite. In addition, $\alpha$ is a $G$-boundary action.
\end{prop}
\begin{proof}
	Let  $O_1, O_2$ be non-empty open sets in $Y$. Suppose $x, y$ are attracting and repelling fixed points of $g$, respectively. First, by minimality of the action, one can find two open neighborhoods $U, V$ of $x, y$, respectively, small enough such that  there are $\gamma_1, \gamma_2\in G$ such that $\gamma_1V\subset O_1$ and $\gamma_2U\subset O_2$.
	Now our assumption on $g$ implies that there is an $m\in \N$ such that $g^m(Y\setminus V)\subset U$, which implies $\gamma_2g^m(Y\setminus V)\subset O_2$. Then one observes that $Y=(\gamma_2g^m)^{-1}O_2\cup \gamma_1^{-1}O_1$. This shows that $\alpha$ is $2$-filling and thus purely infinite. Then $\alpha$ is a strong boundary action and thus is a $G$-boundary action.
\end{proof}

We then mainly focus on the case that a countable discrete group $G$ acting on a proper $\cat$ cube complex $X$.
We first recall the following result in \cite{Ham}. If $G$ acts on a proper $\cat$ space $X$ by isometry. The \textit{limit} set $\Lambda$ of $G$ is the set of accumulation points in $\partial X$ of an orbit of the action, which is closed and $G$-invariant. The action is called \textit{elementary} if either its limit set $\Lambda$ consists of at most two points or if $G$ fixes a point in $\partial X$. 

\begin{prop}\cite[Theorem 1.1]{Ham}\label{prop: minimal on visual}
	Let $G$ acts  by isometry on a proper $\cat$ space $X$ non-elementarily with the limit set $\Lambda$.  Suppose $G$ contains a rank-one isometry. Then $\Lambda$ is perfect and the induced action $\alpha: G\curvearrowright \Lambda$ is minimal. 
\end{prop}

In fact, if the action is cocompact and the visual boundary is non-trivial, we may see more.
\begin{rmk}\label{rmk: non-elementary}
It was proved in \cite{B-B} that if $|\partial X|>2$ and the action of $G$ on $X$ cocompactly by isometry then the action is necessarily non-elementary and the limit set $\Lambda=\partial X$. Note that  it follows that there is no global fixed point in $\partial X$ for $G$ in this case. See more in \cite{B-B} and \cite{Ham}.
\end{rmk}
  
 We denote by $P(G)$ the set of all probability measures on $G$. Let $G$ acts on a metric space $(X, d)$ by isometry. A probability measure $\mu\in P(G)$ is said to have \textit{finite first moment} if $\sum_{g\in G}d(y, gy)\mu(g)<\infty$. A measure $\mu\in P(G)$ is said to be \textit{generating} if the support of $\mu$ generates $\Gamma$ as a semigroup, i.e., for any $g\in G$ there are several $h_1,\dots, h_n\in \supp(\mu)$ such that $g=h_1\dots h_n$. For a $\mu\in P(G)$, we denote by $\check{\mu}$ the reflected measure of $\mu$ on $G$ defined by $\check{\mu}(g)=\mu(g^{-1})$ for $g\in G$. Note that it is very easy to find a generating $\mu\in P(G)$ with finite first moment for a finitely generated group $G$.
Now we have the following result.

\begin{lem}\label{thm: pure inf on visual}
	Let $G$ be a $C^*$-simple finitely generated group and acts properly, freely, essentially and cocompactly by isometry on a proper irreducible finite dimensional $\cat$ cube complex $X$ such that
	\begin{enumerate}
	\item $|\partial X|>2$ and 
	\item  there is a vertex $x\in X^0$ such that the map $g\mapsto gx$ from $G$ to $X$ is a $(c, b)$-quasi-isometric map for some $c, b>0$.
	\end{enumerate}
	Then the reduced crossed product $C(\partial X)\rtimes_r G$ of the induced action $\alpha: G\curvearrowright \partial X$ is simple and purely infinite.
\end{lem}
\begin{proof}
	First, Remark \ref{rmk: non-elementary} implies that  there is no global fixed point of $G$ on $\partial X$.  Then because $X$ is irreducible and the action of $G$ on $X$ by isometry is essential and proper, it follows from \cite[Theorem A]{C-S} that $G$ has a rank-one isometry $g$. Now, Proposition \ref{prop: minimal on visual} implies that $\partial X$ is perfect and the induced action of $G$ on $\partial X$ is minimal.  Note that the rank-one isometry $g$ performs the north-south dynamics on $\partial X$. Therefore, Proposition \ref{prop: north-south visual} implies that $\alpha$ is a $G$-boundary action and thus $2$-filling.
	
	Choose a generating probability measure $\nu$ on $G$ with finite first moment. Denote by $\rho$ the usual word metric on $G$ with respect to a finite generating set. The assumption (2) implies that 
	\[(1/c)\rho(1_G, g)-b\leq d(x, gx)\leq c\rho(1_G, g)+b,\]
	whence for any $n\in \N$, one has 
	$\{g\in G: d(x, gx)\leq n\}\subset B_\rho(1_G, c(b+n))$,
	where $B_\rho(1_G, c(b+n))$ is the ball in $G$ centered at $1_G$ with radius $c(b+n)$. Since $G$ is finitely generated, it has to be at most exponential growth. Therefore, there is a $C>0$ such that 
	\[|\{g\in G: d(x, gx)\leq n\}|\leq |B_\rho(1_G, c(b+n))|\leq e^{Cn}.\]
Then, since the action is additionally minimal, it follows from \cite[Corollary 6.2]{K-M}  that there are two Probability measure $\mu_+$ and $\mu_-$ on the visual boundary $\partial X$ such that  $(\partial X, \mu_+)$ and $(\partial X, \mu_-)$ are Furstenberg-Possion boundaries of $(G, \nu)$ and $(G, \check{\nu})$, which forms a \textit{boundary pair} in the sense of \cite[Definition 2.3]{B-F} by  \cite[Theorem 5.5]{Fer}. Then since the action of $G$ is non-elementary by Remark \ref{rmk: non-elementary}, it follows from \cite[Theorem 7.1]{Fer} that there is a measurable $G$-equivariant map $\varphi: \partial X\to \CR(X)$, which implies that for any $y\in \partial X$ and $g\in G$, if $gy=y$ then $g\varphi(y)=\varphi(y)$. Therefore, the stabilizer group $\stab_G(y)$ is a subgroup of $\stab_G(\varphi(y))$, which is amenable by Theorem \ref{thm: amenable stabilizer on B(X)}. It follows that $\stab_G(y)$ is amenable as well. Now since the  $\partial X$ is a $G$-boundary and $G$ is $C^*$-simple, the action $\alpha$ is topologically free by \cite[Proposition 1.9]{B-K-K-O}. Therefore the reduced crossed product $C(\partial X)\rtimes_r G$ is simple and purely infinite by Theorem \ref{thm: pure inf main}.
\end{proof}
Our main application is still for actions of $G_\Gamma$ on $X_\Gamma$, where $G_\Gamma$ is a RACG or a RAAG while $X_\Gamma$ is the corresponding Davis complex or the universal cover of the Salvetti complex. We have the following main theorem in this subsection.

\begin{thm}\label{cor: main3}
	Let $\Gamma=(V, E)$ be a finite simple graph without joins. 
	\begin{enumerate}
		\item if $|V|\geq 3$ then $C(\partial \Sigma_\Gamma)\rtimes_r W_\Gamma$ is simple and purely infinite.
		\item If $|V|\geq 2$ then $C(\partial \tilde{S_\Gamma})\rtimes_r A_\Gamma$ is simple and purely infinite.
	\end{enumerate}
\end{thm}
\begin{proof}
We still write $G_\Gamma$ for $W_\Gamma$ or $A_\Gamma$ and $X_\Gamma$ for the corresponding complexes for simplicity. For the cases mentioned above, $G_\Gamma$ is finitely generated and $C^*$-simple. In addition, it is known in this irreducible case, the boundary $\partial X_\Gamma$ contains more than $2$ elements. Finally, since the $1$-skeletons of $X_\Gamma$ is  the Cayley graph of $G_\Gamma$, the other conditions of Lemma \ref{thm: pure inf on visual} are easy to be verified.
\end{proof}

\subsection{Actions on Bass-Serre trees}
Another important case involving the visual boundary is that groups act on the visual boundary of trees, especially actions of the fundamental group of a graph of groups on the boundary of its Bass-Serre tree.   We refer to \cite{Ser} and \cite{B-M-P-S-T} for notations and backgrounds. However, we follow the notations in \cite{B-M-P-S-T} and still recall necessary concepts here. Given a graph $\Gamma=(V, E)$, from the viewpoint of groupoids, one may identify the vertex set $V$ with the unit space of $\Gamma$ and the edge set $E$ with ``arrows'' in the groupoids.  Then one may define the \textit{source} and \textit{range} maps of an edge, which provides a direction of each edge.  We abuse the notation by still denoting $E$ for all directed edges. This also allows  to define the ``edge-reversing'' map from $E$ to $E$ by $e\mapsto \bar{e}=e^{-1}$. It is not hard to see $\bar{e}\neq e$, $\bar{\bar{e}}=e$, $s(e)=r(\bar{e})$ and $r(e)=s(\bar{e})$. If we assign indexes for $e$, e.g., $e_i$, then we write $\bar{e}_i$ for its reverse.

\begin{defn}\label{defn: locally finite graph of groups}
A \textit{graph of groups} $\CG=(\Gamma, G)$ consists a connected graph $\Gamma=(V, E)$ and a system of groups:
\begin{enumerate}
	\item a \textit{vertex group} $G_v$ for each $v\in V$;
	\item an \textit{edge group} $G_e$ for each $e\in E$ such that $G_e=G_{\bar{e}}$; and
	\item a monomorphism $\alpha_e: G_e\to G_{r(e)}$ for each $e\in E$.
\end{enumerate}
\end{defn}
For simplicity, we denote by $1_v$ and $1_e$ the identity elements in $G_v$ and $G_e$, respectively. We also write $1$ if the context is clear. A graph $\Gamma=(V, E)$ is said to be \textit{locally finite} if $|r^{-1}(v)|<\infty$ for any $v\in V$. We also say a graph of groups $\CG=(\Gamma, G)$ is locally finite if 
\begin{enumerate}
	\item the underlying graph $\Gamma$ is locally finite; and
	\item $[G_{r(e)}: \alpha_e(G_e)]<\infty$ for any $e\in E$.
\end{enumerate}
The graph of groups $\CG=(\Gamma, G)$ is also called \textit{non-singular} if $[G_{r(e)}: \alpha_e(G_e)]>1$ whenever $r^{-1}(r(e))=\{e\}$. We remark that all graphs $\CG=(\Gamma, G)$ in this subsection are locally finite and non-singular.

\begin{defn}\cite[Definition 2.5]{B-M-P-S-T}
Let $\CG=(\Gamma, G)$ be a graph of groups. The \textit{path group}, denoted by $\pi(\CG)$ is the group generated by the set $E\sqcup \bigsqcup_{v\in V}G_v$ modulo the relations
\begin{enumerate}
	\item[(R1)] $\bar{e}e=1$ for all $e\in E$
	\item[(R2)] $e\alpha_{\bar{e}}(g)\bar{e}=\alpha_e(g)$ for all $e\in E$ and $g\in G_e=G_{\bar{e}}$.
\end{enumerate}
\end{defn}

\begin{defn}\cite[Definition 2.4]{B-M-P-S-T}\label{defn: G-words}
Let $\CG=(\Gamma, G)$ be a graph of groups. For each $e\in E$,  we fix a transversal $\Sigma_e$ for $G_{r(e)}/\alpha_e(G_e)$ with $1_{r(e)}\in \Sigma_e$.
	\begin{enumerate}
\item A $\CG$-\textit{word} (of length $n$) is a sequence of the form
	\[g_1,\ \ \ \text{or } g_1e_1g_2e_2\dots g_ne_n,\ \ \ \text{or } g_1e_1g_2e_2\dots g_ne_ng_{n+1}\]
	such that $s(e_i)=r(e_{i+1})$ for $1\leq i\leq n-1$, $g_j\in G_{r(e_j)}$ for $1\leq j\leq n$ and $g_{n+1}\in G_{s(e_n)}$. In the case $n=0$, the element $g_1\in G_v$ for some $v\in V$.
	\item A reduced $\CG$-word is a $\CG$-word in which if $n>0$ then $g_j\in \Sigma_{e_j}$ for $1\leq j\leq n$ and $g_{i+1}\neq 1_{r(e_{i+1})}$ whenever $e_i=\bar{e}_{i+1}$. Note that there is no restriction for $g_{n+1}\in G_{s(e_n)}$.
	\item A $\CG$-\textit{path} is a reduced $\CG$-word of the form $g_1=1$ or $g_1e_1g_2e_2\dots g_ne_n$.
	\end{enumerate}
 \end{defn}

\begin{defn}\cite[Definition 2.6, 2.7]{B-M-P-S-T}
Let $\CG=(\Gamma, G)$ be a graph of groups. For $v, w\in V$, define $\pi[u, w]\subset \pi(\CG)$ to be the set of images in $\pi(\CG)$ of the $\CG$-words with the range $v$ and the source $w$. In the case that $v=w$, we write $\pi_1(\CG, v)$ for $\pi[v, v]$, which is a subgroup of $\pi(\CG)$. We call $\pi_1(\CG, v)$ the \textit{fundamental group} of $\CG$ based at $v$.
\end{defn}

Note that by using relations (R1) and (R2), the image of any $\CG$-word in $\pi[v, w]$ can be represented by a reduced $\CG$-word. In particular, a typical element in the fundamental group $\pi_1(\CG, v)$ is represented by a reduced $\CG$-word with the source and range $v$.

\begin{defn}\cite[Definition 2.13]{B-M-P-S-T}
Let $\Gamma=(V, E)$ be a graph and $\CG=(\Gamma, G)$ a graph of groups with a base vertex $v\in V$. The \textit{Bass-Serre tree} $X_{\CG, v}$ of $\CG$ has vertex set 
\[X^0_{\CG, v}=\bigsqcup_{w\in V}\pi[v, w]/G_w=\{\gamma G_w: \gamma\in \pi[v, w], w\in V\}.\]
Then there is an edge between vertexes $\gamma G_w$ and $\gamma'G_{w'}$ if $\gamma^{-1}\gamma'\in G_weG_{w'}$ for some $e\in E$ with $r(e)=w$ and $s(e)=w'$.
\end{defn}

It was proved by \cite[Theorem 1.17]{Ba} that $X_{\CG, v}$ is indeed a tree. The natural action of $\pi(\CG, v)$ on $X_{\CG, v}$ is given as follows. Let $\gamma\in \pi_1(\CG, v)=\pi[v, v]$ and $\gamma'G_w\in X^0_{\CG, v}$. One defines $\gamma\cdot \gamma'G_w=\gamma\gamma'G_w$ and this action extends to an action on the edges of $X_{\CG, v}$. We also remark that the Fundamental Theorem of Bass-Serre Theory (see \cite{Ba} and \cite[Theorem 2.16]{B-M-P-S-T}) implies that the whole process above is independent of the choice of the base vertex. Let $v$ be a base vertex in $V$ and then we choose $1_vG_v\in X^0_{\CG, v}$ as our base vertex of $X_{\CG, v}$.

In general, for a tree $T$ with a base vertex $x_0$. By definition, the visual boundary $\partial T$ is exactly the set of infinite branches starting from $x_0$, equipped with cone topology generated by \textit{cylinder sets} of the form $Z(\mu)$. Here, $Z(\mu)$ is the set that consists all infinite branches with a common initial segment $\mu$, where $\mu$ is a finite path starting from $x_0$. One can easily verify that under the cone topology, $\partial T$ is a compact Hausdorff totally disconnected space.

In particular, we denoted by $v\partial X_\CG$ the visual boundary of  $X_{\CG, v}$ with respect to the base vertex $1_vG_v$. We also write $\partial X_\CG=\bigsqcup_{v\in X^0}v\partial X_\CG$ the union  of all boundaries defined from all base vertices of $V$. As one may observe, each vertex of $X_{\CG, v}$ has a unique representative of a $\CG$-path $g_1e_1\dots g_ne_n$ where $r(e_1)=v$. There is an edge between two vertices $\gamma G_v$ and $\gamma'G_{v'}$ if and only if the representative of one of these two vertices extends another with length plus one. See more in \cite[Definition 2.13]{B-M-P-S-T} and \cite[Remark 1.18]{Ba}. From this identification, one may view the visual boundary of $X_{\CG, v}$ by all infinite reduced $\CG$-words with range $v$, i.e., the infinite sequences $g_1e_1g_2e_2\dots$ such that each initial finite subsequence $g_1e_1\dots g_ne_n$ is a reduced $\CG$-word. In addition, the natural induced action of $\pi_1(\CG, v)$ on $v\partial X_\CG$ by homeomorphism can be described in a symbolical way. Let $\gamma=[g_1e_1\dots g_ne_ng_{n+1}]\in \pi_1(\CG, v)$ in which $g_1e_1\dots g_ne_ng_{n+1}$ is a reduced $\CG$-word and $r(e_1)=s(e_{n+1})=v$, and an infinite reduced words $\eta=h_1f_1h_2f_2\dots\in v\partial X_{\CG}$, then one has that  $\gamma\cdot \eta$ is exactly the infinite reduced words uniquely determined by $g_1e_1\dots g_ne_ng_{n+1}h_1f_1h_2f_2\dots$ by doing reduction by using relation (R1) and (R2) even possibly infinite times.  See more in \cite[Section 2.3.1]{B-M-P-S-T}. Now we need the following key concept.

\begin{defn}\cite[Definition 5.14]{B-M-P-S-T}\label{defn: repeatable}
We say a $G$-path $g_1e_1\dots g_ne_n$ is \textit{repeatable} if $r(e_1)=s(e_n)$ and $g_1e_1\neq 1_{r(\bar{e}_n)}\bar{e}_n$. 
\end{defn}

Let $\mu=g_1e_1\dots g_ne_n$ be a repeatable path. Then denote by $\mu^m$ the concatenation of $\mu$ by itself for $m$ times. Note that the repeatability of $\mu$ implies that $\mu^m$ is a reduced word. We also allow $m=\infty$, in which case, $\mu^\infty$ is an infinite reduced  words located in the boundary $v\partial X_{\CG}$.

\begin{prop}\label{prop: repeatable path imply 2-filling}
Let $\Gamma=(V, E)$ be a graph and $\CG=(\Gamma, G)$ a graph of groups. Suppose $v\partial X_{\CG}$ is a infinite set and there is a repeatable word $\mu=g_1e_1\dots g_ne_n$ with $|\Sigma_{\bar{e}_n}|\geq 2$ and the natural action $\beta: \pi(\CG, v)\curvearrowright v\partial X_{\CG}$ is minimal. Then $\beta$ is $2$-filling and thus a strong boundary action.
\end{prop}
\begin{proof}
Let $\mu=g_1e_1\dots g_ne_n$ be the repeatable path. Then let $v=r(e_1)=s(e_n)$ be our base vertex and consider the corresponding fundamental group $\pi_1(\CG, v)$ and the Bass-Serre tree $X_{\CG, v}$, together with the visual boundary $v\partial X_{\CG}$.  Note  that by definition $\mu^m\in \pi_1(\CG, v)$ for any $m\in \N$.

Let $m\in \N$. For any $\CG$-word $sf$ with length $1$ such that $f\in E$ with $r(f)=v$ and $s\in \Sigma_{f}$.  If $f\neq \bar{e}_n$, then for any $\eta\in Z(sf)$, one has that the concatenation ${\mu^m}^{\smallfrown}\eta$ is still an infinite reduced word, whence $\mu^m\cdot Z(sf)\subset Z(\mu^m)$. Define \[A_1=\bigsqcup_{f\in E, r(f)=v, f\neq \bar{e}_n}\bigsqcup_{s\in \Sigma_f}Z(sf)\]
and one actually has $\mu^m\cdot A_1\subset Z(\mu^m)$. Now suppose $f=\bar{e}_n$. Then for any $s\neq 1_{r(\bar{e}_n)}$ and any $\xi\in Z(s\bar{e}_n)$, the concatenation ${\mu^m}^{\smallfrown}\xi$ is still an infinite reduced word. Define 
\[A_2=\bigsqcup_{s\in \Sigma_{\bar{e}_n}\setminus \{1_{r(\bar{e}_n)}\}}Z(s\bar{e}_n)\]
and one has $\mu^m\cdot A_2\subset Z(\mu^m)$. Define $A=A_1\sqcup A_2$. Finally, for the case $B=Z(1_{r(\bar{e}_n)}\bar{e}_n)$, since $|\Sigma_{\bar{e}_n}|\geq 2$, one can choose a $t\in \Sigma_{\bar{e}_n}\setminus \{1_{r(\bar{e}_n)}\}\subset G_{r(\bar{e}_n)}=G_{s(e_n)}$. This shows that $\mu^mt$ is still an group element in $\pi_1(\CG, v)$. Now for any $\zeta\in Z(1_{r(\bar{e}_n)}\bar{e}_n)$, which is of the form $\zeta=1_{r(\bar{e}_n)}\bar{e}_n^{\smallfrown}\rho$ and thus $\mu^mt\cdot \zeta={\mu^m}^{\smallfrown}t\bar{e}_n^{\smallfrown}\rho$, which is a infinite reduced word in $Z(\mu^m)$. This implies that $\mu^mt\cdot B\subset Z(\mu^m)$.

Now observe that $\{Z(\mu^m): m\in \N\}$ forms a neighborhood basis of $\mu^\infty$. For any non-empty open sets $O_1, O_2$ in $v\partial X_\CG$, since the action $\beta$ is minimal, there are $\gamma_1, \gamma_2\in \pi_1(\CG, v)$ and an $m\in \N$ such that $\gamma_i Z(\mu^m)\subset O_i$ for $i=1, 2$.  Now define group elements $h_1=\gamma_1\mu^m$ and $h_2=\gamma_2\mu^mt$ and observe that $h_1A\subset O_1$ and $h_2B\subset O_2$. This implies that the whole boundary $v\partial X_\CG=A\sqcup B\subset h_1^{-1}O_1\cup h_2^{-1}O_2$ and thus $\beta$ is $2$-filling.
\end{proof}

On the other hand, A characterization of minimality of the action $\pi(\CG, v)\curvearrowright v\partial X_{\CG}$ was proved in \cite{B-M-P-S-T}.

\begin{defn}\cite[Definition 5.3]{B-M-P-S-T}\label{defn:flowness}
	Let $\Gamma=(V, E)$ be a graph and $\CG=(\Gamma, G)$ a graph of groups. Let $e, f\in E$. We say $f$ \textit{can flow to} $e$ if $f$ occurs in a infinite reduced words $\xi\in Z(1_{r(e)}e)$ and $f$ is not the rangemost edge of $\xi$.  We say a boundary point $\xi\in \partial X_{\CG}$ can flow to $e$ if there is an $f$ occurs in $\xi$ can flow to $e$.
\end{defn}
See more in \cite[Lemma 5.4]{B-M-P-S-T} for an elementary and more explicit description of Definition \ref{defn:flowness}. Then we record the following theorem on minimality.

\begin{thm}\cite[Theorem 5.5]{B-M-P-S-T}\label{thm: minimal on tree}
	The action  $\beta: \pi(\CG, v)\curvearrowright v\partial X_{\CG}$ is minimal if and only if $\xi$ can flow to $e$ for any $\xi\in \partial X_{\CG}$ and $e\in E$.
\end{thm}

Combining these result, we have the following result immediately.

\begin{thm}\label{thm: boundary action}
	Let $\Gamma=(V, E)$ be a graph and $\CG=(\Gamma, G)$ a graph of groups. Suppose 
	\begin{enumerate}
	\item $v\partial X_{\CG}$ is infinite;
	\item $\xi$ can flow to $e$ for any $\xi\in \partial X_{\CG}$ and $e\in E$; and
	\item there is a repeatable path $\mu=g_1e_1\dots g_ne_n$ with $|\Sigma_{\bar{e}_n}|\geq 2$.
	\end{enumerate}
Then the natural action $\beta: \pi(\CG, v)\curvearrowright v\partial X_{\CG}$ is a strong boundary action. In particular, $\beta$ is a $\pi_1(\CG, v)$-boundary action. If, in addition, each $G_e$  is amenable and $\pi_1(\CG, v)$ is $C^*$-simple, then the action $\beta$ is topologically free and thus the crossed product $C(v\partial X_{\CG})\rtimes_r \pi_1(\CG, v)$ is a unital Kirchberg algebra satisfying the UCT.
\end{thm}
\begin{proof}
By assumptions, observe that $\beta$ is a strong boundary action by Proposition \ref{prop: repeatable path imply 2-filling}  and thus a $\pi_1(\CG, v)$-boundary action. Now if all $G_e$ are amenable groups, then so are all $G_v$ because $[G_{r(e)}, \alpha_e(G_e)]$ is finite for all $e\in E$. Therefore, the action $\beta$ is amenable by \cite[Proposition 5.2.1]{B-O}. Now in the case that  $\pi_1(\CG, v)$ is $C^*$-simple, the crossed product $C(v\partial X_{\CG})\rtimes_r \pi_1(\CG, v)$ is simple and nuclear. Therefore, $\beta$ is topologically free by \cite{A-S}. Then Theorem \ref{thm: pure inf main} implies that $C(v\partial X_{\CG})\rtimes_r \pi_1(\CG, v)$ is purely infinite and thus a unital Kirchberg algebra satisfying the UCT by the classical result of Tu in \cite{Tu}.
\end{proof}

\begin{eg}\label{eg: new1}
Consider the classical case that $\Gamma=(V, E)$ such that $V=\{u, v\}$ and $E=\{e\}$ with $e=(u, v)$. Now define $n_f=|\Sigma_f|=[G_{r(f)}: \alpha_f(G_{f})]$ for $f=e$ or $\bar{e}$. In the \textit{non-degenerated} case, i.e., $(n_e-1)(n_{\bar{e}}-1)\geq 2$, observe that the graph of groups $\CG$ satisfies Theorem \ref{thm: boundary action}. 

Indeed, first, since $n_e, n_{\bar{e}}\geq 2$, one can choose $g, h\neq 1$ such that the path $geh\bar{e}$ is repeatable. For the minimality of the action, since there are only two edges, i.e., $e, \bar{e}$ in $E$, simply observe the infinite reduced words$1e1e\dots$ and $1eh\bar{e}\dots$, where $h\neq 1$ in $Z(1e)$, witness that all $\xi\in \partial X_{\CG}$ flow to $e$ because the only possible edges appeared in $\xi$ are $e$ and $\bar{e}$. The same argument shows all $\xi\in \partial X_{\CG}$ flow to $\bar{e}$ as well. Then it follows from Theorem \ref{thm: minimal on tree} that the action is minimal. On the other hand,  when $n_e$ or $n_{\bar{e}}\geq 3$, the boundary $v\partial X_{\CG}$ is infinite (see e.g., \cite[Example 2.14(E1)]{B-M-P-S-T}). Therefore, the action in Theorem \ref{thm: boundary action} for non-degenerated free product with amalgamation $\pi_1(\CG, v)$ is a strong boundary action. To observe more, if $G_e$ is amenable, then so are $G_v$ and $G_u$ because $n_e, n_{\bar{e}}$ are assumed to be finite in our locally finite setting (Definition \ref{defn: locally finite graph of groups}) and the action $\beta$ is amenable by  \cite[Proposition 5.2.1]{B-O}. Now if $\pi_1(\CG, v)$ is $C^*$-simple, (e.g. when $G_e$ is trivial by \cite[Corollary 12]{Harpe}) then $\beta$ is topologically free. In this case $\pi_1(\CG, v)$, the reduced crossed product $C(v\partial X_{\CG})\rtimes_r \pi_1(\CG, v)$ is a unital Kirchberg algebra satisfying the UCT.
\end{eg}

From now on, we focus on the generalized Baumslag-Solitar groups (GBS groups for simplicity), which can be regarded as  a fundamental group $\pi_1(\CG, v)$ for a graph of groups $\CG=(\Gamma, G)$ in which vertex groups $G_v$ and edge groups $G_e$ are isomorphic to $\Z$. In this case, we also call the graph of groups $\CG=(\Gamma, G)$ a GBS graph of groups.

Let $\CG=(\Gamma, G)$ be a locally finite non-singular GBS graph of groups in which $\Gamma=(V, E)$. Then in Definition \ref{defn: G-words}, one can choose the transversal $\Sigma_e=\{0, 1, \dots, |k_e|-1\}$ for some $k_e\in \Z$ and actually $|k_e|=[G_{r(e)}: \alpha_e(G_e)]$. For each $\CG$-word $\mu=g_1e_1\dots g_ne_ng_{n+1}$, one may assign a rational number $q(\gamma)=\prod_{i=1}^n(k_{\bar{e}_i}/k_{e_i})$ and  verify that the restriction of $q$ on $\pi_1(\CG, v)$ to $\Q^{\times}$ is a group homomorphism.  The graph of groups $\CG$ is called \textit{unimodular} if $|q(\gamma)|=1$ for any $\CG$-word $\gamma$ with $s(\gamma)=r(\gamma)$.

It was also provided in \cite[Theorem 7.5]{B-M-P-S-T} a characterization of when the natural action $\beta: \pi(\CG, v)\curvearrowright v\partial X_{\CG}$ is topologically free. However, if we restrict to finite graphs, one actually has a very nice characterization.

\begin{prop}\cite[Corollary 7.11]{B-M-P-S-T}\label{prop: GBS topological freeness}
Let $\CG=(\Gamma, G)$ be a GBS graph of groups in which $\Gamma=(V, E)$ is a finite graph. Then the natural action $\beta: \pi(\CG, v)\curvearrowright v\partial X_{\CG}$ is topologically free if and only if $\CG$ is not unimodular.
\end{prop}

Combining Theorem \ref{thm: boundary action} and Proposition \ref{prop: GBS topological freeness}, one has the following pure infiniteness result, which also yields a new method to find $C^*$-simple GBS groups.

\begin{thm}\label{thm: main one}
Let $\CG=(\Gamma, G)$ be a  GBS graph of groups in which $\Gamma=(V, E)$ is a finite graph. 
Suppose 
\begin{enumerate}
	\item $v\partial X_{\CG}$ is infinite;
	\item $\xi$ can flow to $e$ for any $\xi\in \partial X_{\CG}$ and $e\in E$;
	\item there is a repeatable path $\mu=g_1e_1\dots g_ne_n$ with $|\Sigma_{\bar{e}_n}|\geq 2$; and
	\item $\CG$ is not unimodular.
\end{enumerate}
Then the natural action $\beta: \pi_1(\CG, v)\curvearrowright v\partial X_{\CG}$ is an topological amenable topologically free strong boundary action and the crossed product $C(v\partial X_{\CG})\rtimes_r \pi_1(\CG, v)$ is a unital Kirchberg algebra satisfying the UCT. Furthermore, $\pi_1(\CG, v)$ is $C^*$-simple.
\end{thm}
\begin{proof}
	Amenability of the action $\beta$ follows from the fact that each vertex group $G_v\simeq \Z$ and \cite[Proposition 5.2.1]{B-O}. Therefore the reduced crossed product is nuclear and thus satisfies the UCT by \cite{Tu}. it follows from Theorems \ref{thm: minimal on tree} and \ref{thm: pure inf main} as well as Proposition \ref{prop: GBS topological freeness} that $C(v\partial X_{\CG})\rtimes_r \pi_1(\CG, v)$ is a unital Kirchberg algebra satisfying the UCT.	For $C^*$-simplicity part, apply \cite[Theorem 1.5]{K-K}.
\end{proof}

One may also want to compare our Theorems \ref{thm: main one} and \ref{thm: boundary action} with results obtained in \cite{B-I-O}.  In particular, we also have the following examples satisfying Theorem \ref{thm: main one}. 

\begin{eg}\label{eg: new2}
First, we claim that Baumslag-Solitar groups (BS groups) $BS(k, l)$  where $(|k|-1)(|l|-1)\geq 2$, satisfies Theorem \ref{thm: main one}. In fact, all $B(k, l)$ can be written as $\pi_1(\CG, v)$, in which $\CG=(\Gamma=(V, E), G)$ such that $V=\{v\}$ and $E=\{e\}$ with $s(e)=r(e)=v$ and $|k|=|\Sigma_e|=[G_{v}: \alpha_e(G_{e})]$ as well as $|l|=|\Sigma_{\bar{e}}|=[G_{v}: \alpha_{\bar{e}}(G_{e})]$.  First, $1e$ is a desirable repeatable element. Then using the same argument in Example \ref{eg: new1}, If $|k| ,|l|\geq 2$ then any $\xi$ flows to any $e$ or $\bar{e}$ and thus the action $\beta$ is minimal. In addition, in this case $v\partial X_{\CG}$ is infinite (see e.g., \cite[Example 2.14(E2)]{B-M-P-S-T}). Finally, $\CG$ is unimodular if and only if $|k|\neq |l|$.
\end{eg}

\begin{eg}($n$-circle)\label{eg: new3}
One may generalize the construction in Example \ref{eg: new2} by considering the GBS graph $\Gamma=(V, E)$ in which $V=\{v_1, \dots, v_n\}$ and $E=\{e_1,\dots, e_n\}$ such that $e_i=(v_{i}, v_{i+1})$ for $1\leq i\leq n-1$ and $e_n=(v_n, v_1)$. Actually, $\Gamma$ is exactly a circle with $n$ vertices and $n$ edges. From this viewpoint, the graph for Example \ref{eg: new2} is exactly a $1$-circle. We now assume each $|k_{e_i}|$ and $|k_{\bar{e}_i}|\geq 2$ and it is not hard to see that in this case the boundary is infinite.

First, a path of the form $g_1e_1\dots g_ne_n$ is repeatable. Then, to verify minimality, similar to the argument in Example \ref{eg: new1}, since only $e_i, \bar{e}_i$ for $i=1, \dots, n$ may occur in any infinite path, for $e_1$, the branch \[1e_11e_2\dots1e_ng\bar{e}_n1\bar{e}_{n-1}\dots1\bar{e}_1\dots,\] where $g\neq 1_{v_1}$, witnesses that any infinite path $\xi\in \partial X_\CG$ flows to $e_1$. By symmetry and the fact that all $|k_{e_i}|$ and $|k_{\bar{e}_i}|\geq 2$, we have $\xi$ flows to $e_i$ for any $\xi\in \partial X_\CG$ and any $e_i$ and $\bar{e}_i$ for $i=1, \dots, n$. This establish the minimality of the action. Now since there is only one circle in the graph $\Gamma$, the $\CG$ is not unimodular if and only if $q=\prod^n_{i=1} k_{e_i}/k_{\bar{e}_i}\neq \pm1$. Therefore, with a proper assignment of all $k_{e_i}$ and $k_{\bar{e}_i}$ for $1\leq i\leq n$ such that 
\begin{enumerate}
\item all $|k_{e_i}|$ and $|k_{\bar{e}_i}|\geq 2$ and
\item $q=\prod^n_{i=1} k_{e_i}/k_{\bar{e}_i}\neq \pm1$, 
\end{enumerate}
the group $\pi_1(\CG, v)$ satisfy Theorem \ref{thm: main one}.
\end{eg}

\begin{eg}(wedge sums of $n$-circles)\label{eg: new4}
	To be even more general, we consider the graph $\Gamma=(V, E)$, which is a wedge sum of two $n$-circles, which means
	\begin{enumerate}
		\item $V=V_1\cup V_2$ such that there is a $v\in V_1\cap V_2$ and $(V_1\setminus\{v\})\cap (V_2\setminus\{v\})=\emptyset$;
		\item $E=E_1\sqcup E_2$ and
		\item the subgraph $\Gamma_i=(V_i, E_i)$ is a $n$-circle for $i=1, 2$.
	\end{enumerate}
Similar work would show that $\pi_1(\CG, v)$ still satisfies Theorem \ref{thm: main one} for proper choices of $k_{e}$ and $k_{\bar{e}}$ for $e\in E$. We also assume all $|k_{e}|$ and $|k_{\bar{e}}|\geq 2$ for $e\in E$. Then, first, each $n$-circle yields a repeatable path and the boundary $v\partial X_\CG$ is infinite.

Now we write $E_1=\{e_1,\dots, e_n\}$ and $E_2=\{f_1,\dots, f_n\}$.  Considering the directions, without loss of any generality, one may assume
\begin{enumerate}
	\item $r(e_1)=r(f_1)=v$;
	\item $s(e_i)=r(e_{i+1})$ and $s(f_i)=r(f_{i+1})$ for $1\leq i\leq n-1$;
	\item $s(e_{n})=r(e_1)$ and $s(f_{n})=r(f_1)$.
\end{enumerate}  
	  Then since edges may appear in any infinite path are just all $e_k, \bar{e}_k$ and $f_l, \bar{f}_l$,  for any $e_i$, the following infinite reduced word \[1e_i1e_{i+1}\dots1e_n1f_1\dots1f_n1e_11e_2\dots1e_{i-1}g\bar{e}_{i-1}1\bar{e}_{i-2}\dots1\bar{e_1}1\bar{f}_n\dots1\bar{f}_11\bar{e}_n\dots1\bar{e}_{i}\dots,\]
	  where $g\neq 1$, witnesses that any infinite branch $\xi\in \partial X_\CG$ flows to $e_i$.  Then the same argument shows that any infinite branch $\xi$ flows to any $e$ or $\bar{e}\in E$. Therefore the action of $\pi_1(\CG, v)$ on $v\partial X_\CG$ is minimal. Finally, for $n$-circles $\Gamma_1$ and $\Gamma_2$,  define $q_1=\prod^n_{i=1} k_{e_i}/k_{\bar{e}_i}$ and $q_2=\prod^n_{i=1} k_{f_i}/k_{\bar{f}_i}$. Then $\CG$ is not unimodular if and only if $q_1$ or $q_2\neq \pm1$ or $q^{\pm 1}_1\cdot q^{\pm1 }_2\neq \pm1$. Combining this, for a wedge sum of $n$-circles $\Gamma=(V, E)$, if we have
	\begin{enumerate}
		\item $|k_e|, |k_{\bar{e}}|\geq 2$ for any $e\in E$ and 
		\item $q_1$ or $q_2$ or $q^{\pm 1}_1\cdot q^{\pm 1}_2\neq \pm 1$;
	\end{enumerate}
then $\pi_1(\CG, v)$ satisfies Theorem \ref{thm: main one}. To deal with wedge sums of $n$-circles for $m$-times, i.e., the underlying graph $\Gamma=(V, E)$ contains $m$ subgraph $\Gamma_i=(V_i, E_i)$, for $1\leq i\leq m$, which are all $n$-circles and there is exactly one common vertex $v$ in all $V_i$ as our base vertex. for each such $\Gamma_i$, we can define corresponding $q_i\in \Q$. Using the same argument, we have if
\begin{enumerate}
	\item $|k_e|, |k_{\bar{e}}|\geq 2$ for all $e\in E$ and 
	\item there is a finite set $F\subset \{1, \dots, m\}$ such that $\prod_{i\in F}q^{\pm 1}_i\neq \pm1$,
\end{enumerate}  
then $\pi_1(\CG, v)$ satisfies Theorem \ref{thm: main one}.
\end{eg}

\begin{rmk}\label{rmk: final}
	We finally remark that GBS group $\pi_1(\CG, v)$ in Example \ref{eg: new3} and \ref{eg: new4} are not isomorphic to a BS group $BS(k, l)$ in Example \ref{eg: new2} whenever $n\geq 2$ or $m\geq 2$.  Even the isomorphism problem of GBS groups have not been fully solved, in our specific cases, there is a way to tell this. We refer readers to \cite{C-F} First, in general, if two non-elementary (see \cite[Section 2.2]{C-F} for its definition) GBS groups $G$ and $H$ are isomorphic, then their corresponding Bass-Serre trees are in the same \textit{deformation space} in the sense of \cite[Definition 2.1]{C-F}, which implies that there is a path from the underlying graph $\Gamma_G$ to $\Gamma_H$ by using moves of three types defined in \cite[Definition 2.2]{C-F}. Moreover, the first Betti numbers of $\Gamma_G$ and $\Gamma_H$ are same.
	
	In our case, first, we consider wedge sums of $n$-circles. Let $\CG_1=(\Gamma_1, G )$ such that $\Gamma_1=(V_1, E_1)$ is a wedge sum of $m$ summands in Example \ref{eg: new4}. Also let $\CG_2=(\Gamma_2, G)$ such that $\Gamma_2$ be the graph of a non-degenerated BS group in Example \ref{eg: new2}. First, the corresponding Bass-Serre trees are non-elemenatry because $\pi_1(\CG_1, v_1)$ are $\pi_1(\CG_2, v_2)$ are not amenable. Then observe their first Betti numbers satisfying $b(\Gamma_1)=m\neq 1=b(\Gamma_2)$, which implies that $\pi_1(\CG_1, v_1)$ is not a non-degenerated BS group.
	
	In the case of $n$-circle, where $n\geq 2$, let $\CG_3=(\Gamma_3, G)$ in which $\Gamma_3$ is the graph of a $n$-circle in Example \ref{eg: new3}. Also the Bass-Serre threes of $\Gamma_3$  and $\Gamma_2$ are non-elementary and reduced in the sense of \cite[Definition 2.1]{C-F}. Then if $\pi_1(\CG_3, v_3)$ is a non-degenerated BS group , then, as we mentioned above, there is a path  of \textit{moves} of the three types defined in \cite[Definition 2.2]{C-F}  from the $\Gamma_2$, which is a $1$-circle to the graph $\Gamma_3$, which is a $n$-circle such that any intermediate graph is still reduced by \cite[Theorem 2.5]{C-F}. However, the only possible first move is a $\mathscr{A}$-move, which happens exactly when the $1$-circle $\Gamma_2$ is a \textit{strict vitural ascending loop}. Otherwise, there is no transformation of  $\Gamma_2$ by using these types of moves. After the first possible $\mathscr{A}$-move, the next moves have to be the cominations of \textit{slide}-moves and induction moves, which however, either keep the graph structure, which is a ``lollipop'' or yield a non-reduced graph. This is a contradiction and thus establish that no groups $\pi_1(\CG_2, v_2)$ of $n$-circles in Example \ref{eg: new3} are BS groups in Example \ref{eg: new2} whenever $n\geq 2$.  
\end{rmk}

We now denote by $\CC$ the class of all groups of graph of groups satisfying Theorems \ref{thm: boundary action} and \ref{thm: main one}. As a summary, we have the following result.

\begin{thm}\label{eg: main4}
Still write $\CG=(\Gamma, G)$, the class $\CC$ includes the following groups.
\begin{enumerate}
\item $C^*$-simple $\pi_1(\CG, v)$ in which  $\CG$ satisfies assumptions (1)-(3) of Theroem \ref{thm: boundary action} and each $G_e$ is amenable. This includes Example \ref{eg: new1} In particular, this includes $G*F$ such that $(|G|-1)(|F|-1)\geq 2$.
\item $C^*$-simple GBS groups $\pi_1(\CG, v)$ appeared in Theorem \ref{thm: main one}.  This includes non-degenerated $BS(k, l)$ where $(|k|-1)(|l|-1)\geq 2$ in Example \ref{eg: new2} and certain GBS groups of $n$-circles in Example \ref{eg: new3} as well as some GBS groups of wedge sum of $n$-circles for $m$ times in Example \ref{eg: new4}. In addition, if $n\geq 2$ or $m\geq 2$, these are not non-degenerated BS groups.
\end{enumerate} 
\end{thm}

We finally remark that all groups in $\CC$ owns a $2$-paradoxical towers in the sense of  \cite{G-G-K-N}, from which, we have the following application.
\begin{thm}\label{thm: final}
Let $\pi_1(\CG, v)$ be the fundamental group appeared in $\CC$, which in particular, contains all examples recorded in Theorem \ref{eg: main4}. Let $H$ be another countable discrete group. Suppose $\pi_1(\CG, v)\times H\curvearrowright X$ is a purely infinite topological amenable minimal topologically free action on a compact metric space $X$ . Then its reduced crossed product is a UCT unital Kirchberg algebra and thus Classifiable by its Elliott invariant.
\end{thm}
\begin{proof}
	Theorem \ref{thm: main one} actually implies that $\pi_1(\CG, v)$ admits $n$-paradoxical towers in the sense of \cite[Definition A]{G-G-K-N}. Then simply apply \cite[Theorem B]{G-G-K-N}.
\end{proof}

\section{Acknowledgement}
The authors should like to thank Johanna Mangahas, Jianchao Wu, Yulan Qing and Abdul Zalloum for helpful discussion and comments. We also would like to thank Tron Omland for useful comments.

\end{document}